\documentclass{ejm}

\let\realverbatim=\verbatim
\let\realendverbatim=\endverbatim
\renewcommand\verbatim{\par\addvspace{6pt plus 2pt minus 1pt}\realverbatim}
\renewcommand\endverbatim{\realendverbatim\addvspace{6pt plus 2pt minus 1pt}}

\ifprodtf \else
  \checkfont{eurm10}
  \iffontfound
    \IfFileExists{upmath.sty}
      {\typeout{^^JFound AMS Euler Roman fonts on the system,
                   using the 'upmath' package.^^J}%
       \usepackage{upmath}}
      {\typeout{^^JFound AMS Euler Roman fonts on the system, but you
                   don't seem to have the}%
       \typeout{'upmath' package installed. EJM.cls can take advantage
                 of these fonts,^^Jif you use 'upmath' package.^^J}%
      }
  \else
  \fi
\fi

\ifprodtf \else
  \checkfont{msam10}
  \iffontfound
    \IfFileExists{amssymb.sty}
      {\typeout{^^JFound AMS Symbol fonts on the system, using the
                'amssymb' package.^^J}%
       \usepackage{amssymb}%
         \let\leq=\leqslant
         \let\geq=\geqslant
      }{}
  \fi
\fi

\ifprodtf \else
  \IfFileExists{amsbsy.sty}
    {\typeout{^^JFound the 'amsbsy' package on the system, using it.^^J}%
     \usepackage{amsbsy}}
    {}
\fi

\newsavebox{\astrutbox}
\sbox{\astrutbox}{\rule[-5pt]{0pt}{20pt}}

\newcommand\eg{e.g.\ }

\renewcommand{\phi}{\varphi}

\newtheorem{theorem}{Theorem}[section]
\newdefinition{definition}[theorem]{Definition}
\newtheorem{lemma}[theorem]{Lemma}
\newtheorem{corollary}[theorem]{Corollary}

\newtheorem{remark}[theorem]{Remark}
\newtheorem{proposition}[theorem]{Proposition}

\usepackage{amsmath}
\usepackage{dsfont}
\usepackage{amsxtra}%
\usepackage{amsfonts}%
\usepackage{enumerate}
\usepackage{amssymb}%
\usepackage{color,hyperref}
\usepackage{graphicx}
\usepackage{subfig}
\usepackage{cite}
\usepackage{tikz}

\hypersetup{colorlinks,breaklinks,
             linkcolor=blue,urlcolor=blue,
             anchorcolor=blue,citecolor=blue}
\usepackage{xcolor}
\usetikzlibrary{arrows.meta}
\usetikzlibrary{shapes}
\usetikzlibrary{plotmarks}

\renewcommand{\L}{{\mathcal L}}
\renewcommand{\O}{{\mathcal O}}
\newcommand{\R}{\mathbb{R}}
\newcommand{\Z}{\mathbb{Z}}

\newcommand{\bv}{{\mathbf v}}
\renewcommand{\r}{{\mathbf r}}
\newcommand{\one}{\mathds{1}}
\newcommand{\ie}{{\it i.e.}}
\newcommand{\M}{{\mathcal M}}
\newcommand{\T}{{\mathbb T}}
\newcommand{\eps}{\varepsilon}
\renewcommand{\div}{\hspace{.1em}\textnormal{div}\hspace{.1em}}

\title[A Continnum Limit for the PageRank Algorithm]{A Continuum Limit for the PageRank Algorithm}

\author[A. Yuan et al.]{%
  A.\ns Y\ls U\ls A\ls N$\,^1$,\ns
  J.\ns C\ls A\ls L\ls D\ls E\ls R$\,^2$\ns
\and
  B.\ns O\ls S\ls T\ls I\ls N\ls G$\,^3$
}

\affiliation{%
  $^1\,$Department of Mathematics, University of Minnesota\\
    email\textup{\nocorr: \texttt{yuanx290@umn.edu}}\\
  $^2\,$Department of Mathematics, University of Minnesota\\
    email\textup{\nocorr: \texttt{jwcalder@umn.edu}}\\
  $^3\,$Department of Mathematics, University of Utah\\
	  email\textup{\nocorr: \texttt{osting@math.utah.edu}}}

\date{\today}
\pubyear{2020}
\volume{000}
\pagerange{\pageref{firstpage}--\pageref{lastpage}}

\begin{document}

\label{firstpage}
\maketitle

\begin{abstract}
Semi-supervised and unsupervised machine learning methods often rely on graphs to model data, prompting research on how theoretical properties of operators on graphs are leveraged in learning problems. While most of the existing literature focuses on undirected graphs, directed graphs are very important in practice, giving models for physical, biological, or transportation networks, among many other applications. In this paper, we propose a new framework for rigorously studying continuum limits of learning algorithms on directed graphs. We use the new framework to study the PageRank algorithm, and show how it can be interpreted as a numerical scheme on a directed graph involving a type of normalized \textit{graph Laplacian}. We show that the corresponding continuum limit problem, which is taken as the number of webpages grows to infinity, is a second-order, possibly degenerate, elliptic equation that contains reaction, diffusion, and advection terms. We prove that the numerical scheme is consistent and stable and compute explicit rates of convergence of the discrete solution to the solution of the continuum limit PDE. We give applications to proving stability and asymptotic regularity of the PageRank vector. Finally, we illustrate our results with numerical experiments and explore an application to data depth. 
\end{abstract}


\begin{subjclass}[2010]
35J15 (Primary); 35D40 (Secondary)
\end{subjclass}
\section{Introduction} 

Due to its versatility in modeling data, graphs are frequently leveraged for applications in  machine learning and data science. A graph structure encodes interdependencies among constituents, such as social media users, images or videos  in a collection, or physical or biological agents, and provides a convenient representation for high dimensional data. For example, in a graph representing research collaborations, we can represent each author as a node in the graph, and co-authorship is represented by edges between nodes, with edge weights depending on the frequency of co-authorship. The resulting graph is \textit{undirected} as the edges are bi-directional. On the other hand, transportation and biological networks often result in \textit{directed} graphs because the relationship between two nodes is ordered, such as the direction of a train route or a predator-prey relationship. 

Graph-based methods are particularly prominent in unsupervised and semi-supervised machine learning tasks that seek to reveal structures and patterns in unlabeled data. For example, in semi-supervised classification, one has labels for a subset of the nodes in the graph, and the problem is to propagate the labels to the rest of the graph in a meaningful way. A widely used and very successful algorithm for semi-supervised classification is Laplacian semi-supervised learning, originally proposed in \cite{zhu2003semi}, which finds the unique \emph{graph harmonic} function that extends the labels. There are many extensions and modifications of Laplacian regularization (see, \eg, \cite{zhou2005learning,zhou2004learning,zhou2004ranking,ando2007learning,Szummer,BelkinMfld}), with more recent methods drawing inspiration from partial differential equations (PDEs) \cite{Bertozzi,garcia2014multiclass}. For classification problems at very low labeling rates, $p$-Laplacian regularization has recently been introduced \cite{ElAlaoui,flores2018}. In unsupervised learning, graph-based algorithms are used in spectral clustering \cite{ng2002spectral,shi2000normalized}, Laplacian eigenmaps \cite{belkin2002laplacian}, diffusion maps \cite{coifman2006diffusion}, manifold ranking~\cite{he2004manifold,he2006generalized,wang2013multi,yang2013saliency,zhou2011iterated,xu2011efficient}, minimal surface graph partitioning \cite{Zosso}, and PageRank \cite{Gleich}.

Various types of graph Laplacians appear in nearly all graph-based learning algorithms, due to the ability of the graph Laplacian to uncover geometric structure in datasets. Graph Laplacians that are commonly used in practice include the \emph{unnormalized} Laplacian
\[\L u(x) = \sum_{y\in X}\omega_{xy}(u(y)-u(x))\]
the \emph{random walk} Laplacian
\[\L^{rw}u(x) =\frac{1}{d_x}\sum_{y\in X}\omega_{xy}(u(y)-u(x)),\]
and the \emph{normalized} Laplacian
\[\L^{n}u(x) =\sum_{y\in X}\frac{\omega_{xy}}{\sqrt{d_xd_y}}u(y)-u(x),\]
where $X$ denotes the set of nodes in the graph, $u: X \to \R$, $\omega_{xy}$ is the (undirected) edge weight between $x$ and $y$, and $d_x = \sum_{y\in X}\omega_{xy}$ is the degree of node $x$. The unnormalized graph Laplacian appears naturally as the gradient of the Dirichlet energy
\[E(u) = \sum_{x,y\in X}\omega_{xy}(u(x)-u(y))^2.\]
The random walk Laplacian is exactly the generator for a random walk on $X$ with probability $d_x^{-1}\omega_{xy}$ of stepping from $x$ to $y$, and the normalized graph Laplacian is a convenient way to obtain a \emph{symmetric} normalization of the graph Laplacian. While these normalizations are most frequently used in practice, many other choices are possible. For example, see \cite{hoffmann2019spectral} for an analysis of how the choice of normalization affects spectral clustering.   We note that the random walk interpretation allows us to view methods like those studied in \cite{zhu2003semi} as performing classification by randomly walking on the graph until hitting a labeled node. Intuitively, the random walk will naturally \emph{learn} the structure of the unlabeled data by remaining within clusters of high density for long enough to hit a labeled point, before moving to a different cluster. While many classification algorithms seek graph harmonic functions, the spectrum of graph Laplacians is widely used to construct low dimensional embeddings of graphs. 

The algorithms discussed above are mainly designed for symmetric graphs, where $\omega_{xy}=\omega_{yx}$. Perhaps one of the most widely known algorithms for directed graphs is the PageRank algorithm, which is used to evaluate the importance of nodes in a graph based on their link structure. While the algorithm is most famous for sorting Google search results up until the mid-$2000$s, variants of PageRank are used by other tech companies (for example, Twitter uses a reversed PageRank to identify influential, topic-specific accounts), and have been adapted to solve problems in neuroscience, genetics, and recommender systems \cite{Gleich}.  The PageRank algorithm uses a random surfer model with teleportation probability $\alpha\in [0,1]$ to rank pages. To describe the model, when the random surfer is at webpage $x$, she will with probability $\alpha$ teleport to a random webpage, and with probability $1-\alpha$ click on an outgoing link to another webpage. When the surfer clicks on an outgoing link, the link is selected at random and we denote by $p_{xy}$ the probability of clicking a link to website $y$ from website $x$. When she randomly teleports, the next website is chosen at random from a teleportation probability distribution $\bv$. The inclusion of the teleportation step ensures the random surfer does not get stuck in disconnected components of the graph. 

The PageRank vector is the invariant distribution of the resulting Markov chain, which measures the amount of time the random surfer spends on each webpage. Webpages that are visited more often by the surfer are ranked more highly, while websites that are rarely visited are ranked lower.  Mathematically, the PageRank vector $\r$ is the (normalized) solution of the eigenvector problem
\begin{align}\label{PageRankEigensystem}
 ((1-\alpha)P + \alpha \bv \one^T)\r = \r, 
\end{align}
where $P = (p_{yx})_{x,y\in X}$ is the probability transition matrix described above, $\bv$ is the teleportation probability distribution, and $\one$ is the column vector of all $1$'s. We note that by the Perron-Frobenius Theorem \cite{Gleich}, the PageRank vector $\r$ can be chosen to have real-valued strictly positive entries. If we choose the normalization $\one^T \r = 1$, so that $\r$ is a probability distribution, then the eigenvector problem \eqref{PageRankEigensystem} is equivalent to the linear system 
\begin{align}\label{PRLinear}
\left(I - (1-\alpha)P \right) \r = \alpha \bv. 
\end{align} 
This formulation is more convenient, since the left hand side can be interpreted as a type of graph Laplacian. 

The teleportation probability distribution $\bv$ can be uniform over all webpages, or can be nonuniform. Indeed, by setting $\bv(x)=\delta_{x_0}(x)$ for a specific website $x_0$ leads to a localized PageRank algorithm that ranks sites nearby $x_0$ \cite{Gleich}. Computationally, the PageRank vector is obtained via the power method on \eqref{PageRankEigensystem}, which converges at a rate of $\left|\lambda_2 / \lambda_1 \right|$, \ie,~a ratio of the second eigenvalue to the leading eigenvalue of the matrix. In the case of PageRank, Haveliwala and Kamvar \cite{Haveliwala} show that $\lambda_1=1$ and $\lambda_2=1-\alpha$, so the convergence rate depends heavily on the choice of the teleportation parameter. Google takes $1-\alpha$ to be $.85$ \cite{Langville}. There are also adaptations of semi-supervised learning to directed graphs (see \cite{zhou2005semi}). 

Due to the ubiquity of graph Laplacians in graph-based learning problems, much work has been devoted to understanding how the graph Laplacian is able to uncover geometric and distributional structure from unlabeled data. To do this, one usually assumes the graph is a random geometric graph with $n$ points and length scale $h>0$, and considers the limit as $n\to \infty$ and $h\to 0$. This means the nodes in the graph are an \emph{i.i.d.}~sample of size $n$ from a density $\rho$ supported on a $d$-dimensional manifold $\M$ embedded in $\R^D$, and the weights $\omega_{xy}$ are defined by
\[\omega_{xy} = \Phi\left( \frac{|x-y|}{h} \right),\]
where $\Phi:[0,\infty)\to [0,\infty)$ is nonincreasing and usually compactly supported. The first results to appear in the literature were pointwise consistency results, showing that a graph Laplacian $\L$ applied to a smooth test function $\phi\in C^3(\M)$ converges, as $n\to \infty$ and $h\to 0$ to a weighted version of the Laplace-Beltrami operator
\[\Delta_\rho \phi = \rho^a \div(\rho^b \nabla (\rho^c \phi))\]
for various values of $a,b,c$ that depend on the choice of normalization of the graph Laplacian. For example, for the unnormalized graph Laplacian, $a=-1, b=2, c=0$, and for the random walk Laplacian $a=-2,b=2,c=0$. If $h\to 0$ and $n\to \infty$ simultaneously, then the condition $nh^{d+2}\gg \log n$ is required for pointwise consistency, which ensures there are enough neighbors of each data point to apply appropriate concentration of measure results. To obtain $O(h)$ pointwise consistency rates, it is required that $nh^{d+4}\gg 1$. We contrast this with the condition $nh^d\gg \log n$ required for graph connectivity. For pointwise consistency results of this flavor, see \cite{Bousquet,LafonThesis,Hein,HeinMore,BelkinUniform,Singer}. Pointwise consistency was extended to $k$-nearest neighbor graph constructions in \cite{Ting}, which includes some mildly directed graphs due to antisymmetries in the $k$-nearest neighbor relation.

While pointwise consistency results are informative, they do not prove that the solutions of graph-based problems converge to solutions of their counterparts as $n\to \infty$ and $h\to 0$. This question is more subtle and requires further analysis. The problem of spectral convergence of the graph Laplacian spectrum to that of the Laplace-Beltrami operator has been well-studied. Belkin and Niyogi \cite{BelkinEigen} established $L^2$ spectral convergence (convergence of eigenvalues and $L^2$ convergence of eigenvectors) when $\rho$ is the uniform distribution and this was extended to non-uniform distributions with partial convergence rates in  \cite{von2008consistency}. Shi \cite{Shi} proved convergence rates and extended the analysis to include manifolds with boundary. The $L^2$ convergence rate was improved recently in \cite{GarciaEigen} using variational methods, and then further improved in \cite{calder2019improved} to agree with the pointwise consistency rate $O(h)$, which is the sharpest known $L^2$ spectral convergence rate. The variational parts of the spectral convergence arguments in \cite{calder2019improved,GarciaEigen} were heavily influenced by earlier work in a non-probabilistic setting \cite{Burago}. We also mention that very recent work has established the first $L^\infty$ eigenvector convergence rates \cite{dunson2019diffusion}.

For problems in clustering and semi-supervised learning, recent work has begun to address convergence in the continuum using tools from the calculus of variations and viscosity solutions of PDEs. Trillos and Slep\v cev \cite{SlepcevTV} developed a Gamma-convergence framework for proving discrete to continuum convergence of graph-based problems, and the framework has been applied to prove discrete to continuum consistency in many problems (see, \eg \cite{SlepcevTV,trillos2016consistency,trillos2018variational,trillos2019knn,SlepcevLp,Osting} and the references therein). Discrete to continuum convergence can also be established with the maximum principle and the viscosity solution framework, as was established in \cite{calder2018game,calder2019consistency} for the game-theoretic graph $p$-Laplacian and Lipschitz learning. The maximum principle can also be used to prove asymptotic H\"older regularity of solutions to graph-based learning problems, as was done in  \cite{calder2019properly,calder2018game}. For the linear $2$-graph Laplacian, \cite{RyanNicolas2019} used the maximum principle to establish discrete to continuum convergence rates for regression problems, and \cite{Calder2020rates} used the maximum principle in combination with random walk arguments to establish convergence rates for semi-supervised learning at low labeling rates. We also mention that \cite{shi2018error} uses the maximum principle to prove convergence rates for a reweighted version of the graph Laplacian in low label rate semi-supervised learning context. 

Despite the flurry of recent work on discrete to continuum consistency results, almost none of the results apply to problems on \emph{directed graphs}, which are important and widely used in practice. The only results we are aware of for directed graphs are for $k$-nearest neighbor graphs \cite{Ting,trillos2019knn}, which are directed only due to the asymmetry of the $k$-nearest neighbor relation.  Discrete to continuum results are important for providing insights and further understanding of algorithms. Furthermore, continuum limits allow us to prove stability of graph-based algorithms, showing that they are insensitive to the particular realization of the data, and often can lead to new formulations of learning problems founded on stronger theoretical principles. This paper aims to start filling this void by studying consistency results for problems on \emph{directed graphs}. We propose the \emph{random directed geometric graph} model, which extends the random geometric graph in a natural way by adding directionality in the weights. For concreteness, we study the PageRank problem, and prove that the PageRank vector converges in the large sample size limit to the solution of a continuum, possibly degenerate, elliptic PDE. Depending on the strength of the directionality in the weights, the continuum PDE can be a first order equation, which is a new type of result for consistency of graph Laplacians. Our main results are finite sample size error estimates with high probability, which imply convergence in the continuum, but are stronger in that they hold in the non-asymptotic regime. We use these results to prove stability of the PageRank problem, and we also study the time-dependent version of the problem, which examines the continuum limit of the distribution of the random surfer. Our proofs use pointwise consistency and the maximum principle, with appropriate adaptations to directed graphs. We also present the results of some numerical experiments confirming our theoretical results, and exploring applications to data depth.

\section{Setup and main results}
\label{sec:setup}

We now describe our setup and main results. Section~\ref{sec:RDGG} introduces our \emph{random directed geometric graph} model, and Section~\ref{sec:main} formulates the PageRank problem in a new way and gives our main results.

\subsection{Random directed geometric graph}
\label{sec:RDGG} 

In order to study continuum limits for problems on directed graphs, we propose a new model for a random directed graph that we call a \emph{random directed geometric graph}. Let $x_1, x_2, \ldots, x_n$ be an \emph{i.i.d.}~sample of size $n$ on the torus $\mathbb{T}^d=\R^d/\Z^d$ with a density function $\rho:\T^d \to [0,\infty)$. We define the weight $\omega_{xy}$ from $x$ to $y$ by
\begin{equation*}
\omega_{xy} = \Phi\left(\frac{\left|B(x)(y-x-\varepsilon b(x))\right|}{h}\right), 
\end{equation*}
where $B:\T^d\to \R^{d\times d}$ and $b:\T^d\to \R^d$, and $B(x)$ has full rank for every $x\in \T^d$. The parameter $h>0$ is the bandwidth of the kernel, and $\eps>0$ is the strength of the directionality. The kernel function  $\Phi$ is assumed to be nonnegative with compact support. When $B=I$ and $b=0$ or $\eps=0$, the weights are the same as those of a random geometric graph, which is symmetric. For other choices of $B$ and $b$, the corresponding graph is directed, with directional influence along the vector field $b$. The matrix $B$ can be viewed as changing the metric locally. 

Assume for the moment that $\Phi$ has compact support in $[0,2]$. We observe that for fixed $x$, the support of the weight $y\mapsto \omega_{yx}$ is the ellipse shaped region 
\begin{align}
E_x := \left\{ y \in \mathbb{R}^d \hspace{.1cm} \Big| \hspace{.1cm} \frac{|B(x)(y-x-\varepsilon b(x))|}{h} < 2 \right\}\label{domain} 
\end{align} 
which depends on $x$. Figure~\ref{weightsDomain} gives us a sense of $E_x$ in two dimensions.  In the random walk (or random surfer) interpretation, the random walker moves from $x$ to a point in the set $E_x$, which contains a drift term $\eps b(x)$ and an anisotropic diffusion governed by $B$.

\begin{figure}
\centering
\begin{tikzpicture}[scale=1.25]
\draw[thick,blue] (0,0) -- (4,2);
\filldraw (0,0) circle (2pt);
\node[anchor = north west] at (0,0) {$x$};
\filldraw (4,2) circle (2pt);
\node[anchor = north west] at (4,2) {$x+\eps b(x)$};
\draw[rotate around={-45:(4,2)},thick] (4,2) ellipse (2 and 1);
\draw[rotate around={45:(4,2)},-Latex,thick,red] (4,2) -- (5,2);
\draw[rotate around={135:(4,2)},-Latex,thick,red] (4,2) -- (6,2);
\node[anchor = south west] at (4.25,2) {$r_1 h$};
\node[anchor = south east] at (4,2.5) {$r_2 h$};
\end{tikzpicture}

\caption{A visualization of the set $E_x$, which is the support of the weight $y\mapsto\omega_{xy}$ in $\mathbb{R}^2$. The blue line represents the directional preference $b(y)$ multiplied by $\varepsilon$. The ellipse is the set $E_x$, where $r_i = 2/\sqrt{\lambda_i}$ and $\lambda_1, \lambda_2$ are the eigenvalues of $B(x)^T B(x)$, and the red arrows indicate the eigenvectors of $B(x)^TB(x)$.}
\label{weightsDomain}
\end{figure}
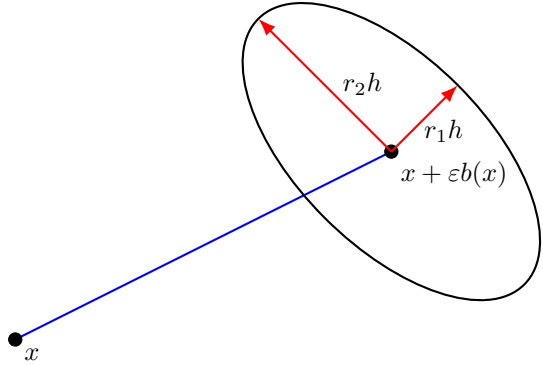

In the following remarks, we provide an in-depth motivation for the weights in the context of ranking players in a sports tournament and modeling systematic distortion in the data \cite{Shnitzer}.

\begin{remark}[Motivation via ranking]
Assume each team or player is represented by a feature vector $x\in \R^d$ that adequately describes each player. When player $x$ and player $y$ play against each other, we write $x\succ y$ if $x$ wins the game, and $y \succ x$ if $y$ wins. We assume each time $x$ and $y$ play, $x$ wins with probability $\mathbb{P}(x\succ y)$ and $y$ wins with probability $\mathbb{P}(y\succ x) = 1-\mathbb{P}(x \succ y).$

Suppose that $x$ and $y$ play $n$ games, and we assign an edge in our graph from $x$ to $y$ if $y$ wins more than half of the games, and assign the edge from $y$ to $x$ if $x$ wins more than half the games. That is, the edge is directed toward the ``better'' player. The weight on the edge is the excess number of wins for the winning player. In expectation, if $\mathbb{P}(x\succ y)\geq \tfrac{1}{2}$, then we have an edge from $y$ to $x$ with weight $$\omega_{yx}= -1 + 2\mathbb{P}(x\succ y).$$
If $\mathbb{P}(y \succ x) \geq \tfrac{1}{2}$ then we have an edge from $x$ to $y$ with weight $$\omega_{xy}= -1 + 2\mathbb{P}(y\succ x) = 1 - 2\mathbb{P}(x\succ y).$$
Now, we make a modeling assumption on $\mathbb{P}(x\succ y)$. We assume there is some underlying (unknown) ranking function
\[\phi:\R^d \to [0,1],\]
so that $\phi(x)\geq r(y)$ indicates player $x$ is better than $y$. A natural model for the probability $\mathbb{P}(x\succ y)$ is then $$\mathbb{P}(x\succ y) =\frac{1}{2} + \frac{\phi(x) - \phi(y)}{2},$$
leading to the weight
\[\omega_{yx}= \phi(x) - \phi(y),\]
when $\phi(x)\geq \phi(y)$. Using a Taylor expansion, we have an edge from $y$ to $x$ if 
\begin{align*}\phi(y) \leq \phi(x) \approx \phi(y) + \nabla \phi(y)\cdot (x-y) +\frac{1}{2}(x-y)^T \nabla^2 \phi (x-y),
\end{align*}
or
\begin{align} \nabla \phi(y)\cdot (x-y) +\frac{1}{2}(x-y)^T \nabla^2 \phi (x-y) \geq 0. \label{ellipse} \end{align} 

If $\phi$ is convex, the set of $x$ satisfying the inequality above lie in an ellipse. This gives some motivation for the directional preference $b=\nabla \phi$ and for the elliptical shape governed by $B=(\nabla^2 \phi)^{1/2}$ as they occur in the definition of our weights. We restrict the weights locally to some ball $B(y,2h)$ based on the assumption that teams play against similarly ranked teams in a tournament. 
\end{remark}

\begin{remark}
The work in \cite{Shnitzer} constructs two operators that can identify the common structures and the differences, respectively, between two diffeomorphic Riemannian manifolds. An application to identifying signals from fetal electrocardiogram (EGC) data via observed maternal ECG data is considered. Their algorithm handles cases where there is a systematic diffusion in the observed vs.~target data; our problem is ``adjacent'' in the sense that we model the diffusion and directional preferences via $B(y)$ and $b(y)$ in our setup. 
\end{remark}

\subsection{Main results}
\label{sec:main}

We now present our setup and main results. We take the random directed geometric graph model from Section~\ref{sec:RDGG} with $B(x)\equiv I$ (though see Section~\ref{GeneralizedB} for a discussion of how the results change when $B$ is not the identity). That is, let $x_1,x_2,\dots,x_n$ be an \emph{i.i.d.}~sample of size $n$ on the torus $\T^d$ with probability density $\rho:\T^d\to [\rho_{min},\infty)$, where $\rho_{min}>0$, and set $X_n = \{x_1,x_2,\dots,x_n\}$. We define the weight from $x$ to $y$ by
\begin{equation}\label{eq:weights}
\omega_n(x,y) = \Phi\left(\frac{\left|y-x-\varepsilon b(x)\right|}{h}\right), 
\end{equation}
and the degree of $x$ by $d_n(x) = \sum_{y\in X_n}\omega_n(x,y)$. We assume the kernel $\Phi$ is smooth, compactly supported on $[0,2]$, nonnegative, nonincreasing,  satisfies $\Phi(0) > 0$ and
\begin{equation}\label{eq:massone}
\int_{B(0,2)}\Phi(|z|) \, dz =1.
\end{equation}

As constructed, for example in \cite{Szummer}, the probability of a random walk on the graph transitioning from $x$ to $y$ is $p_{xy}= d_n(x)^{-1}\omega_n(x,y)$. Plugging this into \eqref{PRLinear} and denoting the PageRank vector by $r_n:X_n\to \R$, we find that $r_n$ satisfies the linear system
\begin{align}\label{PRproblem}
r_n(x) - (1-\alpha) \sum_{y \in X_n} \frac{\omega_{n}(y,x)}{d_n(y)} r_n(y) = \alpha v(x) \ \ \ \text{ for all }x \in X_n,
\end{align} 
where $\alpha\in (0,1]$ is the teleportation probability and $v(x)$ is the teleportation probability distribution. To simplify the problem, we consider the normalized PageRank vector 
\begin{equation}\label{eq:PRnorm}
u_n(x) = \frac{nh^d}{d_n(x)}r_n(x).
\end{equation}
The degree $d_n(x)$ is the most basic measure of the importance of a node in a graph, and the normalized PageRank vector factors out the direct dependence on the degree to give an understanding of the additional geometric structure uncovered by PageRank. We easily see that the normalized PageRank vector $u_n:X_n\to \R$ satisfies the equation
\begin{align}\label{PRproblem_simpler}
d_n(x)u_n(x) - (1-\alpha) \sum_{y \in X_n} \omega_n(y,x) u_n(y) = \alpha nh^d v(x) \ \ \ \text{ for all }x \in X_n.
\end{align} 
We note that \eqref{PRproblem_simpler} is considerably simpler to analyze than \eqref{PRproblem} since the degree term $d_n(y)$ does not appear inside the summation. We rewrite this equation by defining the \textbf{PageRank Operator}
\begin{align} 
\L_n u(x) := \frac{1}{d_n(x)}\sum_{y \in X_n} \omega_n(y,x) u(y) - u(x).
\label{PageRankOp} 
\end{align}
Then \eqref{PRproblem_simpler} can be written as
\begin{equation}\label{PRScheme}
u_n(x) - \gamma \L_n u_n(x) = \frac{nh^d}{d_n(x)}v(x)  \ \ \ \text{for all }x\in X_n,
\end{equation}
where $\gamma = (1-\alpha)/\alpha$. We note that when the graph is symmetric, the PageRank Operator is exactly the random walk graph Laplacian. 

The corresponding problem in the continuum is the, possibly degenerate, elliptic PDE
\begin{align}\label{PRPDE} 
u +\gamma_\eps\rho^{-2}\div(\rho^2b u)- \frac{1}{2}\sigma_\Phi\gamma_h\rho^{-2} \div(\rho^2\nabla u) = \rho^{-1}v \ \ \ \textrm{ on } \mathbb{T}^d, 
\end{align} 
where $\sigma_\Phi= \int_{\R^d} \Phi(|z|)z_1^2 dz$,
\begin{equation}\label{eq:Gammas}
\gamma_\eps = \frac{(1-\alpha)\eps}{\alpha} \ \ \text{ and }\ \ \gamma_h = \frac{(1-\alpha)h^2}{\alpha}.
\end{equation}
We also denote 
\begin{equation}\label{AssumptionforRate}
\eta = \|\rho^{-2}\div(\rho^2b)\|_{L^\infty(\T^d)}.
\end{equation}
Our first main result is the following continuum limit.
\begin{theorem}[Convergence of the $2^{\rm nd}$ Order PageRank Problem]\label{Rate}
Let $\rho\in C^{2,\nu}(\T^d)$, $b\in C^{2,\nu}(\T^d;\R^d)$ and $v\in C^{1,\nu}(\T^d)$ for some $0 < \nu < 1$. Assume that $\gamma_\eps\leq 1$, $0 < \gamma_h \leq 1$, and $\eta < 1$. Let $u_n$ be the solution to the PageRank problem \eqref{PRScheme} and let $u\in C^{3}(\T^d)$ be the solution to the PDE \eqref{PRPDE}. Then there exists $C_1,C_2,c_1,c_2>0$ with $C_1$ depending on $\gamma_h>0$, such that when $\eps + h\leq c_1(1-\eta\gamma_\eps)$ we have that
\begin{equation}\label{ConvRate}
\max_{x\in X_n}| u(x) - u_n(x)| \leq C_1 (1-\eta\gamma_\eps)^{-1}(\lambda + \eps + h)
\end{equation}
holds with probability at least $1- C_2n\exp(-c_2nh^{d+2}\lambda^2) - C_2n\exp\left( -c_2nh^{d+2}(1-\eta\gamma_\eps)^2 \right)$, where $0 < \lambda \leq 1$.
\end{theorem}

\begin{remark}
We remark that when $\gamma_h>0$ and $\eta\gamma_\eps<1$, it is a standard result in elliptic PDEs that \eqref{PRPDE} has a unique solution $u\in C^{3,\nu}(\T^d)$. We refer the reader to \cite{Evans,Gilbarg} for more details.
\label{rem:PDE}
\end{remark}

When $\gamma_h=0$ or $\gamma_h>0$ is small, the continuum PDE \eqref{PRPDE} is dominated by the first order terms, and is better approximated by the first order equation
\begin{align}\label{PROp1stOrder} 
u +\gamma_\eps\rho^{-2}\div(\rho^2b u)= \rho^{-1}v \ \ \ \textrm{ on } \mathbb{T}^d.
\end{align} 
We state the first order continuum limit as a separate result.
\begin{theorem}[Convergence of the $1^{\rm st}$ Order PageRank Problem]\label{1stRate}
Let $\rho\in C^{1,1}(\T^d)$, $b\in C^{1,1}(\T^d;\R^d)$, and $v\in C^{0,1}(\T^d)$. Assume that $\gamma_\eps,\gamma_h \leq 1$, $\eta < 1$, and $\|Db\|_{L^\infty(\T^d)}\leq \tfrac{1}{2}(1-\eta\gamma_\eps)$. Let $u_n$ be the solution to the PageRank problem \eqref{PRScheme} and let $u\in C^{0,1}(\T^d)$ be the viscosity solution of the PDE \eqref{PROp1stOrder}. Then there exists $C_1,C_2,c_1>0$ such that
\begin{equation}\label{FirstConvRate}
\max_{x\in X_n}| u(x) - u_n(x)| \leq C_1\sqrt{\lambda + \eps + \gamma_h}
\end{equation}
holds with probability at least $1- C_2n\exp(-c_1nh^{d+2}\lambda^2) - C_2n\exp\left( -c_1nh^{d+2}(1-\eta\gamma_\eps)^2 \right)$, where $0 < \lambda \leq 1$. We note that $C_1$ depends on $1-\eta\gamma_\eps$.
\end{theorem}
\begin{remark}
When $\eta\gamma_\eps<1$, it is a standard result in viscosity solution theory that  \eqref{PROp1stOrder} has a unique viscosity solution $u\in C(\T^d)$. We prove in Lemma~\ref{lem:Lip} that when $\|Db\|_{L^\infty(\T^d)}\leq \tfrac{1}{2}(1-\eta\gamma_\eps)$ the viscosity solution $u$ is Lipschitz continuous, so $u\in C^{0,1}(\T^d)$. We refer the reader to \cite{Crandall,calder2018lecture} for more details on viscosity solutions.

We note that the continuum PDE \eqref{PRPDE} has reaction, advection, and diffusion terms. The two reaction terms, $u$ and $\rho^{-1}v$, are due to the teleportation step in PageRank. The term $ \div(\rho^2b u)$ is an advection term, which describes the advection of the quantity $\rho^2 u$ along the vector field $b$, and is due to the directional preference in the definition of the weights in a random directed geometric graph. Finally, the weighted diffusion term $\div(\rho^2 \nabla u)$ represents diffusion from the random walk step of PageRank.

Theorem~\ref{Rate} shows that if we scale $\eps_n\sim \alpha_n$ and $h_n\sim \sqrt{\alpha_n}$, then the directional preference along the vector field $b$ is balanced with the diffusion terms, and the limiting PDE \eqref{PRPDE} is second order. If $\eps_n\ll \alpha_n$, then the directional preference is negligible in the limit, and the first order terms in \eqref{PRPDE} disappear in the limit. Theorem~\ref{1stRate} shows that if $h_n\ll \sqrt{\alpha_n}$, then the directional preference term dominates and the diffusion term is negligible, and the continuum PDE reduces to a first order equation \eqref{PROp1stOrder}. If $\alpha_n \gg \max\{\eps_n,h_n^2\}$, then the first and second order terms drop out of \eqref{PRPDE} and we simply get $u = \rho^{-1}v$. The intuition is that the teleportation happens too often and overwhelms the diffusion and directional preferences, yielding a trivial continuum limit. 
\label{rem:reaction}
\end{remark}
\begin{remark}
We can analyze the characteristics of the first order equation \eqref{PROp1stOrder} to understand how the PageRank algorithm propagates information on the directed graph. The characteristic ODEs \cite{Evans} are
\begin{equation}\label{eq:characteristics}
\left\{\begin{aligned}
\dot{p}(s) &=z(s) \nabla (\div b + 2\nabla \log \rho\cdot b)\Big\vert_{x(s)} + Db(x(s))p(s),\\ 
\dot{z}(s) &= b(x(s)) \cdot p(s),\\ 
\dot{x}(s) &= b(x(s)),
\end{aligned}\right.
\end{equation}
where $x(s)$ is the projected characteristic curve, $z(s) = u(x(s))$, and  $p(s) = \nabla u(x(s))$. Hence, information is propagated along the integral curves of the vector field $b$, which represents the directional influence in the random directed geometric graph.
\label{rem:characteristics}
\end{remark}

\begin{remark}
Theorems~\ref{Rate} and \ref{1stRate} are stated as finite sample size results, where $n$, $\eps$, $h$, $\alpha$, and $\lambda$ are fixed. If we consider the continuum limit as $n\to \infty$ and $\eps_n,h_n,\alpha_n,\lambda_n\to 0$, then Theorems~\ref{Rate} and \ref{1stRate} inform us about how to relatively scale the parameters. We always assume $\eps_n\leq \alpha_n$ and $h_n^2 \leq \alpha_n$, so that $\gamma_{\eps_n},\gamma_{h_n}\leq 1$.  Thus, provided that
\begin{equation}\label{eq:almostsure}
\lim_{n\to \infty}\frac{nh_n^{d+2}\lambda_n^2}{\log n}= \infty,
\end{equation}
we may apply the Borel-Cantelli Lemma to conclude that the rates hold almost surely as $n\to \infty$. This allows us to make a suitable choice for $\lambda_n$. Since the convergence rates scale with $\lambda_n$, we choose $\lambda_n\to 0$ as quickly as possible while ensuring that \eqref{eq:almostsure} holds. A reasonable choice is 
\begin{equation}\label{eq:lambdan}
\lambda_n = \frac{\log n}{\sqrt{nh_n^{d+2}}}.
\end{equation}
With this choice of $\lambda_n$, we have the almost sure convergence rate of
\[\O\left( \frac{\log (n)^2}{\sqrt{nh_n^{d+2}}} + \eps_n + h_n \right)\]
in Theorem \ref{Rate}, provided 
\begin{equation}\label{eq:almostsure2}
\lim_{n\to \infty}\frac{nh_n^{d+2}}{\log (n)^2}= \infty.
\end{equation}
The scaling in \eqref{eq:almostsure2} is a standard scaling for pointwise consistency of graph Laplacians. Another way to phrase these results is to make the choice of 
\[\lambda_n = \eps_n+h_n\]
to match the terms in the error estimate in Theorem \ref{Rate}. In this case, we have an almost sure convergence rate of $\O(\eps_n+h_n)$  provided
\begin{equation}\label{eq:almostsure_rate}
\lim_{n\to \infty} \frac{nh_n^{d+2}(\eps_n+h_n)^2}{\log(n)}=\infty.
\end{equation}
The same observations hold in the  context of Theorem \ref{1stRate}, except the rates are worse by a square root. 
\label{rem:scaling}
\end{remark}
\begin{remark}\label{rem:continuum}
Theorems~\ref{Rate} and \ref{1stRate} can easily be rewritten in terms of the true PageRank vector $r_n(x)$. Due to Lemma~\ref{ConMeas1} we have
\[\max_{x\in X_n}| \rho(x) u(x) - r_n(x)| \leq C_1 (1-\eta\gamma_\eps)^{-1}(\lambda + \eps + h)\]
in the context of Theorem~\ref{Rate}, and 
\[\max_{x\in X_n}| \rho(x) u(x) - r_n(x)| \leq C_1\sqrt{\lambda + \eps + \gamma_h}\]
in the context of Theorem~\ref{1stRate}.

The PageRank vector $r_n$ satisfies 
\[\sum_{x\in X_n}r_n(x) = \sum_{x\in X_n}v(x),\]
as can be easily checked by summing both sides of \eqref{PRproblem}. This forces $r_n$ to be a probability distribution provided $v$ is as well. The normalized PageRank vector $u_n$ satisfies
\[\sum_{x\in X_n} \frac{d_n(x)}{ nh^d} u_n(x) = \sum_{x\in X_n}v(x).\]
In the continuum, the solution $u$ of \eqref{PRPDE} or \eqref{PROp1stOrder} satisfies the analogous continuum version 
\[\int_{\T^d}\rho^2 u \, dx = \int_{\T^d}\rho v\, dx,\]
which can be verified by multiplying both sides of \eqref{PRPDE} or \eqref{PROp1stOrder} by $\rho^2$ and integrating by parts.
\end{remark}

\begin{remark}
While the original PageRank problem is an eigenvector problem, the PageRank vector is an eigenvector of a probability transition matrix and not an eigenvector of a graph Laplacian. Thus, we cannot use the spectral properties of the graph Laplacian proven, for example, in \cite{Shi,GarciaEigen,calder2019improved} to address the eigenvalue problem \eqref{PageRankEigensystem}. In fact, since the probability transition matrix becomes localized as $h,\eps\to 0$, we lose the interpretation of PageRank as an eigenvector problem in the continuum. The localization of the probability transition matrix is exactly what leads to a equation with a Laplacian in the continuum. A very simple analogue is the averaging operator
\[T_\eps u(x) := \frac{1}{|B(x,\eps)|}\int_{B(x,\eps)}u(y)\, dy,\]
A function $u$ satisfying $T_\eps u = u$ is an eigenfunction of $T_\eps$ with eigenvalue $\lambda=1$. In PDE-theory, the equation $T_\eps u = u$ is called the mean-value property, and is satisfied by any harmonic function $u$. One can easily check that
\[\frac{1}{\eps^2}(T_\eps u(x) - u(x)) = C\Delta u(x) + O(\eps)\]
for any smooth function $u$, where $C$ depends only on $d$. Hence, as $\eps\to 0$, eigenfunctions of $T_\eps$ are expected to converge to harmonic functions, which are solutions of Laplace's equation $\Delta u =0$. The operator $T_\eps$ localizes and becomes trivial as $\eps\to 0$, since it reduces to pointwise evaluation  $T_0 u(x)=u(x)$. Thus, there is no meaningful way to think of harmonic functions as eigenfunctions of $T_0$. An analogous, but more complicated, phenomenon occurs with the PageRank problem, as it also becomes localized in the continuum limit. 
\end{remark}

As an immediate application of Theorems~\ref{Rate} and \ref{1stRate} we prove asymptotic Lipschitz regularity of the PageRank vector, which shows that the ranking does not vary rapidly in feature space.
\begin{corollary}[Lipschitz regularity]\label{cor:LIP}
Under the assumptions of Theorem~\ref{Rate}, for $0 < \lambda \leq 1$ and  with probability at least $1- C_2n\exp(-c_2nh^{d+2}\lambda^2) - C_2n\exp\left( -c_2nh^{d+2}(1-\eta\gamma_\eps)^2 \right)$ we have
\begin{equation}\label{eq:LIP}
| u_n(x) - u_n(y)| \leq C|x-y| + C_1 (1-\eta\gamma_\eps)^{-1}(\lambda + \eps + h)
\end{equation}
for all $x,y\in X_n$.
\end{corollary}
\begin{proof}
By the triangle inequality
\[|u_n(x)-u_n(y)| \leq |u_n(x) - u(x)| + |u(x) - u(y)| + |u_n(y) -u(y)|.\] 
We estimate the first and third term with Theorem~\ref{Rate}, while the second is estimated by Lipschitzness of $u$.
\end{proof}
\begin{remark}
Corollary~\ref{cor:LIP} proves that $u_n$ is approximately Lipschitz continuous, with jumps of size no larger than $O(\lambda+\eps+h)$. We note that an analogous result to Corollary~\ref{cor:LIP} can be stated under the assumptions of Theorem~\ref{1stRate} as well.
\end{remark}

We conclude this section by presenting an analogous continuum limit result for the evolution of the probability distribution of the random surfer $\r^k$. Similar to Eq.~\eqref{PageRankEigensystem}, the probability distribution $\r^k$ of the random surfer satisfies the evolution equation 
\begin{align}\label{PageRankTransition}
\r^{k+1}= ((1-\alpha)P + \alpha \bv \one^T)\r^k.
\end{align}
Since $\r^k$ is a probability distribution, so $\one^T\r^k=1$, we can also write the equation as
\begin{align}\label{PageRankTransition2}
\r^{k+1}= (1-\alpha)P\r^k + \alpha \bv.
\end{align}
As before,  we denote by $r_n(x,k)$ the $x$-component of $\r^k$; that is, $r_n(x,k)$ is the probability of finding the random surfer at vertex $x$ after $k$ steps on the random directed geometric graph of size $n$.  Plugging this into \eqref{PageRankTransition2} we find that $r_n(x,k)$ satisfies 
\begin{align}\label{PRproblem_time}
r_n(x,k+1) = (1-\alpha) \sum_{y \in X_n} \frac{\omega_{n}(y,x)}{d_n(y)} r_n(y,k) + \alpha v(x) \ \ \ \text{ for all }x \in X_n,
\end{align} 
where $v(x)$ is the teleportation probability distribution. As before, we simplify the problem by defining the normalized distribution $u_n(x,k) := \frac{nh^d}{d_n(x)}r_n(x,k)$ and find that $u_n(x,k)$ satisfies
\begin{equation}\label{eq:PRevolution}
\frac{u_n(x,k+1)-u_n(x,k)}{\alpha} +u_n(x,k) - \gamma\L_n u_n(x,k) = \frac{nh^d}{d_n(x)}v(x) \ \ \text{ for all }x\in X_n,
\end{equation}
For the initial condition we take $u_n(x,0)=g(x)$ for some given smooth function $g$. We can think of \eqref{eq:PRevolution} as a discrete heat equation on the graph, describing the evolution of the normalized distribution $u_n$ of the random surfer. The stationary point of the evolution, as $k\to \infty$, is clearly the solution of the PageRank problem \eqref{PRScheme}.

The continuum version of \eqref{eq:PRevolution} is the reaction-advection-diffusion equation
\begin{equation}\label{eq:continuum}
\left\{\begin{aligned}
u_t + u +\gamma_\eps\rho^{-2}\div(\rho^2bu)- \frac{1}{2}\sigma_\Phi\gamma_h\rho^{-2} \div(\rho^2\nabla u) &= \rho^{-1}v,&&\text{in }\T^d\times \{t>0\}\\ 
u &=g,&&\text{on } \T^d\times \{t=0\}.
\end{aligned}\right.
\end{equation}
This is verified by the following continuum limit result.
\begin{theorem}[Continuum limit for random surfer]\label{TimeRate}
Let $\rho\in C^{2,\nu} (\T^d)$, $b\in C^{2,\nu} (\T^d;\R^d)$, $v\in C^{1,\nu} (\T^d)$, and $g\in C^3(\T^d)$ for some $0 < \nu < 1$. Assume that $\gamma_\eps\leq 1$, $0 < \gamma_h \leq 1$, and $\eta < 1$. Let $u_n(x,k)$ be the solution of \eqref{eq:PRevolution} satisfying $u_n(x,0)=g(x)$, and let $u\in C^{3}(\T^d)$ be the solution to the PDE \eqref{eq:continuum}. Then there exists $C_1,C_2,c_1,c_2>0$ with $C_1$ depending on $\gamma_h>0$, such that when $\eps + h\leq c_1(1-\eta\gamma_\eps)$ and $0< \lambda \leq 1$, the event that 
\begin{equation}\label{TimeConvRate}
\max_{x\in X_n}| u(x,\alpha k) - u_n(x,k)| \leq C_1 \alpha k (\lambda + \eps + h)
\end{equation}
holds for all $k\geq 0$ has probability at least 
\[1- C_2n\exp(-c_2nh^{d+2}\lambda^2) - C_2n\exp\left( -c_2nh^{d+2}(1-\eta\gamma_\eps)^2 \right).\] 
\end{theorem} 
Theorem~\ref{TimeRate} shows that the distribution of the random surfer can be approximated by the continuum PDE \eqref{eq:continuum}. The error estimates depend on $\lambda,\eps$ and $h$ in a similar way as in Theorem~\ref{Rate}. The main difference is the appearance of the term $k\alpha$, which corresponds to the time parameter in the continuum PDE \eqref{eq:continuum}, and is due to the accumulation of pointwise consistency errors over $k$ steps. 

\begin{remark}
A first order version of Theorem \ref{TimeRate} can be proved, similar to Theorem \ref{1stRate}. Since the statement and proof are very similar to Theorem \ref{1stRate}, we omit the details.
\label{rem:1sttime}
\end{remark}
\begin{remark}
We mention that another interesting perspective is the inverse problem of using graph-based numerical schemes, like the PageRank scheme, to numerically solve the continuum reaction-advection-diffusion equations. Modulo technical details, all of the results in this paper can be extended to the manifold setting, where $\T^d$ is replaced by a smooth compact and connected manifold $\mathcal{M}$ of dimension $d$ embedded in $\R^D$ where $d < D$. Since graph-based numerical schemes learn the geometry of the manifold automatically, they may provide convenient numerical methods for solving PDEs on manifolds. We mention that ideas along these lines were mentioned in  \cite{Burago} for approximating the spectrum of the Laplace-Beltrami operator on a manifold of dimension higher than $2$ or $3$, where finite-element methods become cumbersome. This is an interesting direction to explore in future research.
\label{rem:PDEsolvers}
\end{remark}

\subsection{Outline}
The paper will be organized as follows. In Section~\ref{consistency} we prove pointwise consistency with high probability for the PageRank operator $\L_n$ on a random directed geometric graph. This includes pointwise consistency to both first and second order continuum operators. In Section~\ref{ConvergenceProofs} we prove our main results, Theorems~\ref{Rate}, \ref{1stRate}, and \ref{TimeRate}. Lastly, in Section~\ref{Numerical}, we present some numerical results to support our arguments. 
 
\section{Consistency for \texorpdfstring{$\L_n$}{Ln}}\label{consistency}

In this section, we prove pointwise consistency for the operator $\L_n$ with both first and second order continuum operators. Throughout this section and the rest of the paper, we write $\omega_{xy}=\omega_n(x,y)$ and $d_x=d_n(x)$ for simplicity.

\subsection{Concentration of measure and change of variables}
We first recall a concentration inequality from \cite{calder2018game}. 
\begin{lemma}[{\cite[Remark 7]{calder2018game}}]\label{ConMeas}
Let $Y_1, Y_2, \dots, Y_n$ be a sequence of i.i.d.~random variables on $\mathbb{R}^d$ with density $f: \mathbb{R}^d \to \mathbb{R}$, let $\psi: \mathbb{R}^d \to \mathbb{R}$ be bounded and Borel measurable with compact support in a bounded open set $\Omega \subset \mathbb{R}^d$, and define $$Y= \sum_{i=1}^n \psi(Y_i).$$ Then for any $0 \leq \lambda \leq 1$, 
\begin{align} 
\mathbb{P} \Big[|Y-\mathbb{E}(Y)| > \|f\|_\infty \|\psi\|_\infty n |\Omega| \lambda \Big] \leq 2 \exp\left(-\frac{1}{4} \|f\|_\infty n |\Omega| \lambda^2 \right), \label{ConMeasFormat} 
\end{align}
where $|\Omega|$ denotes the Lebesgue measure of $\Omega$. 
\end{lemma}

When we apply the lemma to the degree $d_y$, we need to compute the expected value of $d_y$, which is the integral 
\[\mathbb{E}(d_y) = n \int_{E_y} \Phi\left(\frac{|x-y-\varepsilon b(y)|}{h}\right) \rho(x) dx.\]
For this computation, and others, we require asymptotic expansions in the change of variables formulas, which is provided by the following result.

\begin{lemma}[Change of Variables]\label{ApplyCoV}
Let $g:\R^d\to \R$ be continuous and assume $b$ is $C^1$. Then using the change of variables $z = \frac{x-y-\varepsilon b(y)}{h}$, we have
\begin{equation}\label{eq:intx}
\int_{\R^d} \Phi\left(\frac{|x-y-\varepsilon b(y)|}{h}\right) g(x) dx = h^d\int_{\R^d} \Phi\left(|z|\right) g(y+hz + \eps b(y)) dz,
\end{equation}
and
\begin{align}\label{eq:inty}
&\int_{\R^d} \Phi\left(\frac{|x-y-\varepsilon b(y)|}{h}\right) g(y) dy \\
&= h^d\int_{\R^d} \Phi\left(|z|\right) g(x-hz-\eps b(x) + \mathcal{O}(\eps h + \eps^2)) \left(1-\varepsilon \div b(x) + \mathcal{O}(\varepsilon h + \varepsilon^2)\right) dz,\notag
\end{align}
for sufficiently small $\eps>0$.
\end{lemma} 
\vspace{-8mm}
\begin{proof}
The proof is a straightforward change of variables when we integrate over $x$ in \eqref{eq:intx}. In the case where we integrate over $y$ in \eqref{eq:inty}, we define the change of variables function $G:\R^d\to \R^d$ by
\[G(y) = \frac{x-y-\eps b(y)}{h},\]
where the vector field $b$ is extended periodically from  $\T^d$ to $\R^d$. For $\eps$ sufficiently small, $G$ is a $C^1$ diffeomorphism of $\R^d$. Indeed, to see that $G$ is bijective, first note that the problem of solving $G(y)=z$ for $y$, given $z$, is equivalent to the fixed point problem
\[y = \Psi(y):= x - zh - \eps b(y).\]
The mapping $\Psi:\R^d\to \R^d$ is a contraction, provided $\eps <\|b\|_{C^{0,1}}^{-1}$, and so the invertibility of $G$ follows from Banach's fixed point theorem. For $\eps <\|b\|_{C^{0,1}}^{-1}$ the Jacobian matrix $D_y G(y) = -h^{-1}(I + \eps D_y b(y))$ is invertible, and so the smoothness of $G^{-1}$ follows from the inverse function theorem.

We now use the change of variables $z = G(y)$ in \eqref{eq:inty} to obtain
\begin{equation}\label{eq:change}
\int_{\R^d} \Phi\left(\frac{|x-y-\varepsilon b(y)|}{h}\right) g(y) dy = \int_{\R^d} \Phi\left(|z|\right) g(G^{-1}(z)) |\det D_y G(G^{-1}(z))|^{-1} dz.
\end{equation}
In the rest of the proof we write $y=G^{-1}(z)$ for convenience. Then we have
\begin{equation}\label{eq:jacobian}
|\det D_y G(y)|^{-1} = h^d |\det(I+\eps Db(y))|^{-1}.
\end{equation}
We now use the Taylor expansion
\[\det(I + \eps A) = 1 + \eps\text{Tr}(A) + \O(\eps^2)\]
for $A=Db(y)$ to obtain
\[\det(I + \eps Db(y)) = 1 + \eps \text{Tr}(Db(y)) + \O(\eps^2) = 1 + \eps\div b(x) + \O(\eps h + \eps^2).\]
Thus, for sufficiently small $\eps>0$ we have
\begin{equation}\label{eq:Jtaylor}
|\det D_y G(y)|^{-1} = h^d(1 - \eps\div b(x) + \O(\eps h + \eps^2)).
\end{equation}
Since $\Phi(t)=0$ for $t>2$ we have that $|x-y| \leq 2h + \eps |b(y) | \leq C(h+\eps)$,
and so
\begin{equation}\label{eq:ychange}
G^{-1}(z) = y = x - h z - \eps b(y) =  x-hz - \eps b(x) + \O(\eps h+\eps^2).
\end{equation}
Substituting  \eqref{eq:Jtaylor} and \eqref{eq:ychange} into \eqref{eq:change} completes the proof.
\end{proof}

\subsection{Pointwise consistency}

We now turn to the main pointwise consistency results. We begin with a standard result for the degree.
\begin{lemma}[Asymptotics for the degree] \label{ConMeas1}
For any $0 < \lambda \leq 1$, the degree term satisfies 
\begin{align*}
d_y &= nh^d  \rho(y) + \mathcal{O}(nh^d(\lambda + \varepsilon + h^2))
\end{align*}
with probability at least $1-2\exp(-c\lambda^2 n h^d)$, where $c$ is a constant independent of $n$. 
\end{lemma} 
\begin{proof}
We use part $(i)$ in Proposition~\ref{ApplyCoV} to compute
\begin{align*}
\mathbb{E}(d_y) = \mathbb{E}\left(\sum_x \omega_{yx} \right) &= n \int_{E_y} \Phi\left(\frac{|x-y-\varepsilon b(y)|}{h}\right) \rho(x) dx \\ 
&= nh^d \int_{B(0,2)}\Phi(|z|) \rho\left(y+hz+ \varepsilon b(y)\right) dz \\ 
&= nh^d\left( \rho(y) + \mathcal{O}(\varepsilon + h^2)\right).  
\end{align*}
By Lemma~\ref{ConMeas}, we see that 
\begin{align*}
\mathbb{P} \left[ \Big| d_y - \mathbb{E}[d_y]\Big| > C\lambda nh^d \right] \leq 2\exp(-\frac{1}{4} cnh^d \lambda^2)
\end{align*}
for any $0 \leq \lambda \leq 1$. Combining the observations above completes the proof.
\end{proof}

We now state and prove the consistency result. 
\begin{theorem}[Consistency for the $2^{\rm nd}$ order PageRank Operator]\label{ConsistencyResult}
 There exists constants $C, c >0$ such that for any $0 < \lambda \leq 1$ the event that
\begin{align}\label{eq:consistency}
\frac{d_x}{\rho(x)nh^d}\L_n \varphi(x) &=-\rho^{-2}\div(\rho^2b\phi)\eps + \frac{\sigma_\Phi}{2}\rho^{-2} \div(\rho^2\nabla \varphi) h^2\Big|_x\\
&\hspace{1in}+\O\left((\lambda h + \lambda h^{-1}\eps + \varepsilon^2 + h^3 + \eps h) \|\varphi\|_{C^{2,1}(\T^d)} \right) \notag
\end{align}
holds for all $\phi\in C^{2,1}(\T^d)$ and $x\in X_n$ has probability at least $1-Cn\exp(-cnh^d \lambda^2)$. 
\end{theorem}
\begin{proof}
Fix $x\in \T^d$, take $\varphi \in C^{2,1}(\mathbb{R}^d)$ to be a test function, and let $p = D\varphi(x)$ and $a_{ij} = \varphi_{x_ix_j}(x)$. We apply the operator $\L_n$ to $\varphi$ at $x$ and take a second-order Taylor expansion at $x$ of the $\varphi(y)$ inside the summation, which gives us
\begin{align*}
d_x\L_n \varphi(x) &=\sum_{y\in X_n\setminus \{x\}} (\omega_{yx} \varphi(y) - \omega_{xy}\phi(x))\nonumber\\
&= \sum_{y\in X_n\setminus \{x\}}(\omega_{yx} - \omega_{xy})\phi(x) +\sum_{i=1}^d p_i\sum_{y\in X_n\setminus \{x\}}\omega_{yx}(y_i-x_i)\\ 
&\hspace{0.35in} + \frac{1}{2} \sum_{i,j=1}^da_{ij}\hspace{-2mm} \sum_{y\in X_n\setminus \{x\}} \hspace{-2mm}\omega_{yx} (y_i - x_i)(y_j-x_j) + \O\left((\varepsilon^3 + h^3) \beta\hspace{-2mm}\sum_{y\in X_n\setminus \{x\}} \omega_{yx} \right),
\end{align*}
where $\beta=\|\phi\|_{C^{2,1}(\T^d)}$, and the $\varepsilon^3 + h^3$ in the remainder term comes from the scaling of $|y-x|$. Since $\varphi(x)$ and its derivatives are factored out from the summations, the probability estimates that follow are independent of $\phi$ and hold uniformly over all smooth test functions. 

Noting that $|\omega_{xy}-\omega_{yx}|\leq Ch^{-1}\eps$, we have by Lemma~\ref{ConMeas} that each of
\begin{align*}
\sum_{y\in X_n\setminus \{x\}} \omega_{yx}&\leq  C nh^d,\\
\frac{1}{n}\hspace{-1mm}\sum_{y\in X_n\setminus \{x\}} (\omega_{yx}-\omega_{xy}) &=\int_{\T^d} (\omega_{yx}-\omega_{xy})\rho(y)dy +\mathcal{O}\left(\lambda h^{-1}\eps\right),  \\
\frac{1}{n}\hspace{-1mm}\sum_{y\in X_n\setminus \{x\}} \omega_{yx} (y_i -x_i) &=\int_{\T^d} \omega_{yx}(y_i-x_i)\rho(y) dy +\mathcal{O}\left(\lambda(\eps+h)\right), \textrm{ and } \\
\frac{1}{n}\hspace{-1mm}\sum_{y\in X_n\setminus \{x\}} \omega_{yx} (y_i -x_i)(y_j-x_j) &=  \int_{\T^d} \omega_{yx} (y_i -x_i)(y_j-x_j)\rho(y)dy  + \mathcal{O}\left(\lambda(\eps^2+h^2)\right),
\end{align*} 
hold with probability at least $1-2\exp\left( -cnh^d\lambda^2 \right)$ for any $0< \lambda \leq 1$. 
Combining the observations above we have with probability at least $1-C\exp\left( -cnh^d\lambda^2 \right)$ that
\begin{align}\label{eq:nonlocal}
\frac{d_x}{nh^d}\L_n \varphi(x) &=\frac{1}{h^d}\int_{\T^d} (\omega_{yx}-\omega_{xy})\rho(y)\, dy \phi(x) + \sum_{i=1}^d\frac{p_i}{h^d}\int_{\T^d}\omega_{yx}(y_i-x_i)\rho(y) \, dy\\
&\hspace{1in}+ \frac{1}{2} \sum_{i,j=1}^d\frac{a_{ij}}{h^d} \int_{\T^d} \omega_{yx} (y_i - x_i)(y_j-x_j)\rho(y)\, dy\notag \\
&\hspace{2in}+\O\left(\lambda(h + h^{-1}\eps)\beta + (\varepsilon^3 + h^3) \beta \right). \notag
\end{align}

We now compute asymptotic expansions for all the terms in \eqref{eq:nonlocal}. By Lemma~\ref{ApplyCoV} we have
\begin{align*}
\frac{1}{h^d}&\int_{\T^d}\omega_{yx}\rho(y)\, dy\\
&=\frac{1}{h^d}\int_{\T^d}\Phi\left(\frac{\left|x-y-\varepsilon b(y)\right|}{h}\right) \rho(y)\, dy\\
&=\underbrace{\int_{B(0,2)}\Phi(|z|)\rho(x - hz - \eps b(x) + \O(\eps h + \eps^2))\, dz}_{A}\, (1 - \eps \div b(x)  + \O(\eps h + \eps^2)).
\end{align*}
We now compute
\begin{align*}
A&=\int_{B(0,2)}\Phi(|z|)\left( \rho(x) - \nabla \rho(x)\cdot (hz+\eps b(x)) + \frac{h^2}{2}z^T\nabla^2 \rho(x)z + \O(\eps h + \eps^2)\right)\, dz\\
&=\rho(x) - \nabla \rho(x)\cdot b(x)\eps + \frac{\sigma_\Phi }{2}\Delta \rho(x) h^2 + \O(\eps h + \eps^2).
\end{align*}
Therefore
\begin{align*}
\frac{1}{h^d}&\int_{\T^d}\omega_{yx}\rho(y)\, dy\\
&=\left(\rho(x) - \nabla \rho(x)\cdot b(x)\eps + \frac{\sigma_\Phi }{2}\Delta \rho(x) h^2 + \O(\eps h + \eps^2)\right)(1 - \eps \div b(x)  + \O(\eps h + \eps^2))\\
&=\rho(x) - \nabla \rho(x)\cdot b(x)\eps + \frac{\sigma_\Phi }{2}\Delta \rho(x) h^2 - \rho(x)\div b(x)\eps + \O(\eps h + \eps^2).
\end{align*}
By Lemma~\ref{ApplyCoV}  we have
\begin{align*}
\frac{1}{h^d}&\int_{\T^d}\omega_{xy}\rho(y)\, dy\\
&=\frac{1}{h^d}\int_{\T^d}\Phi\left(\frac{\left|y-x-\varepsilon b(x)\right|}{h}\right) \rho(y)\, dy\\
&=\int_{B(0,2)}\Phi(|z|)\rho(x + hz + \eps b(x))\, dz\\
&=\int_{B(0,2)}\Phi(|z|)\left(\rho(x) +  \nabla \rho(x)\cdot (hz + \eps b(x)) + \frac{h^2}{2}z^T \nabla^2\rho(x)z + \O(\eps h+\eps^2)\right)\, dz\\
&=\rho(x) + \nabla \rho(x)\cdot b(x)\eps + \frac{\sigma_\Phi}{2}\Delta \rho(x) h^2 + \O(\eps h+\eps^2).
\end{align*}
Therefore
\begin{equation}\label{eq:first}
\frac{1}{h^d}\int_{\T^d} (\omega_{yx}-\omega_{xy})\rho(y)\, dy = -2\nabla \rho(x)\cdot b(x)\eps - \rho(x)\div b(x) \eps + \O(\eps h + \eps^2).
\end{equation}

By Lemma~\ref{ApplyCoV}  again we have
\begin{align}\label{eq:second}
&\frac{1}{h^d}\int_{\T^d}\omega_{yx}(y_i-x_i)\rho(y)\, dy\\
&=\frac{1}{h^d}\int_{\T^d}\Phi\left(\frac{\left|x-y-\varepsilon b(y)\right|}{h}\right) (y_i-x_i)\rho(y)\, dy\notag \\
&=-\int_{B(0,2)}\hspace{-3mm}\Phi(|z|)\rho(x - hz +\O(\eps))(z_i h + b_i(x)\eps + O(\eps h +\eps^2))\, dz\, (1 + \O(\eps))\notag \\
&=-\int_{B(0,2)}\hspace{-3mm}\Phi(|z|)\left(\rho(x) - \nabla \rho(x)\cdot z h +\O(\eps+h^2)\right)(z_i h + b_i(x)\eps + O(\eps h +\eps^2))\, dz\notag \\
&=-\int_{B(0,2)}\hspace{-3mm}\Phi(|z|)\left(\rho(x)z_i h + \rho(x)b_i(x)\eps - (\nabla \rho(x)\cdot z)z_i h^2 +\O(\eps h+h^3+\eps^2)\right)\, dz\notag \\
&=-\rho(x)b_i(x)\eps + \sigma_\Phi \rho_{x_i}(x) h^2 +\O(\eps h+h^3+\eps^2).\notag
\end{align}

Finally, another application of Lemma~\ref{ApplyCoV}  yields
\begin{align}\label{eq:third}
\frac{1}{h^d}&\int_{\T^d}\omega_{yx}(y_i-x_i)(y_j-x_j)\rho(y)\, dy\\
&=\frac{1}{h^d}\int_{\T^d}\Phi\left(\frac{\left|x-y-\varepsilon b(y)\right|}{h}\right) (y_i-x_i)(y_j-x_j)\rho(y)\, dy\notag \\
&=\int_{B(0,2)}\hspace{-3mm}\Phi(|z|)(\rho(x) +O(h+\eps))(z_iz_j h^2+ O(\eps h +\eps^2))\, dz\notag \\
&=\sigma_\Phi\rho(x)\delta_{ij} h^2 + \O(h^3 + \eps h + \eps^2),\notag
\end{align}
where $\delta_{ij}=1$ if $i=j$ and $\delta_{ij}=0$ otherwise. Combining \eqref{eq:first}, \eqref{eq:second}, and \eqref{eq:third} with \eqref{eq:nonlocal} we have
\begin{align}\label{eq:local}
\frac{d_x}{nh^d}\L_n \varphi(x) &=-(2\nabla \rho(x)\cdot b(x) + \rho(x)\div b(x) )\phi(x) \eps\\
&\hspace{0.75in}+ \nabla \varphi(x)\cdot (\sigma_\Phi \nabla \rho(x)h^2-\rho(x)b(x)\eps)+ \frac{\sigma_\Phi}{2} \rho(x)\Delta \varphi(x) h^2\notag \\
&\hspace{1.5in}+\O\left(\lambda(h + h^{-1}\eps)\beta + (\varepsilon^2 + h^3 + \eps h) \beta \right). \notag
\end{align}

We now divide both sides of  \eqref{eq:local} by $\rho(x)$ and use the identities 
\[\rho^{-2}\div(\rho^2b \phi) =\phi(\div b + 2\nabla \log \rho \cdot b) + \nabla \varphi \cdot b,\] 
and
\[\frac{1}{2}\rho^{-2} \div(\rho^2\nabla \varphi) = \frac{1}{2}\Delta \varphi + \nabla \log \rho \cdot \nabla \varphi.\]
to establish that the event that \eqref{eq:consistency} holds for all $\phi \in C^{2,1}(\T^d)$ and a fixed $x\in \T^d$ has probability at least $1-C\exp\left( -cnh^d\lambda^2 \right)$.

To establish that \eqref{eq:consistency} holds for all $x\in X_n$, we first condition on the random variable $x_1$. The remaining points $X_n\setminus \{x_1\}=\{x_2,x_3,\dots,x_{n-1}\}$ are an independent \emph{i.i.d.}~sample of size $n-1$, and we can write, as at the start of the proof, that
\[d_{x_1}\L_n \phi(x_1) = \sum_{y\in X_n\setminus \{x_1\}}(\omega_{yx_1}\phi(y) - \omega_{x_1y}\varphi(x_1)).\]
Since the sum on the right is over $n-1$ \emph{i.i.d.}~random variables, we can apply the same argument as above, and the law of conditional probability, to establish that the event that \eqref{eq:consistency} holds for $x=x_1$ and all $\phi\in C^{2,1}(\T^d)$ has probability at least $1-C\exp\left( -c(n-1)h^d\lambda^2 \right)$. We then repeat the argument, conditioning on each of $x_2,x_3,\dots,x_n$, and union bound over all $n$ points to obtain that \eqref{eq:consistency} holds for all $x\in X_n$ and $\phi \in C^{2,1}(\T^d)$ with probability at least $1-Cn\exp\left( -c(n-1)h^d\lambda^2 \right)$. The proof is completed by using that $n-1 \geq n/2$ for $n\geq 2$ to simplify the probability. 
\end{proof}

We also have a corresponding consistency result when the continuum PDE is first order.
\begin{theorem}[Consistency for the $1^{\rm st}$ order PageRank Operator]\label{FirstConsistencyResult}
 There exists constants $C, c >0$ such that for any  $0 < \lambda \leq 1$ the event that
\begin{equation}\label{eq:firstconsistency}
\frac{d_x}{\rho(x)nh^d}\L_n \varphi(x) =-\rho^{-2}\div(\rho^2b\phi)\eps+\O\left((\lambda h + \lambda h^{-1}\eps + \varepsilon^2 + h^2) \|\varphi\|_{C^{1,1}(\T^d)} \right)
\end{equation}
holds for all $\phi\in C^{1,1}(\T^d)$ and $x\in X_n$ has probability at least $1-Cn\exp(-cnh^d \lambda^2)$. 
\end{theorem}
\begin{proof}
The proof follows closely to that of Theorem~\ref{ConsistencyResult}, so we sketch it here.
Fix $x\in \T^d$, take $\varphi \in C^{1,1}(\mathbb{R}^d)$ to be a test function, and let $p = D\varphi(x)$. We have
\begin{align*}
d_x\L_n \varphi(x) &= \sum_{y\in X_n} \omega_{yx} \varphi(y) - d_x\phi(x) \nonumber\\
&= \sum_{y\in X_n}(\omega_{yx} - \omega_{xy})\phi(x) +\sum_{i=1}^d p_i\sum_{y\in X_n}\omega_{yx}(y_i-x_i) +\O\left((\varepsilon^2 + h^2) \beta\sum_y \omega_{yx} \right),
\end{align*}
where $\beta=\|\phi\|_{C^{1,1}(\T^d)}$. Thus, with probability at least $1-C\exp\left( -cnh^d\lambda^2 \right)$ we have that
\begin{align}\label{eq:nonlocal_first}
\frac{d_x}{nh^d}\L_n \varphi(x) &=\frac{1}{h^d}\int_{\T^d} (\omega_{yx}-\omega_{xy})\rho(y)\, dy \phi(x) + \sum_{i=1}^d\frac{p_i}{h^d}\int_{\T^d}\omega_{yx}(y_i-x_i)\rho(y) \, dy\\
&\hspace{2in}+\O\left(\lambda(h + h^{-1}\eps)\beta + (\varepsilon^2 + h^2) \beta \right). \notag
\end{align}
By \eqref{eq:first} and \eqref{eq:second} we have 
\begin{align*}
\frac{d_x}{nh^d}\L_n \varphi(x) &=-(2\nabla \rho(x)\cdot b(x) + \rho(x)\div b(x) )\phi(x) \eps- \nabla \varphi(x)\cdot \rho(x)b(x)\eps\notag \\
&\hspace{2in}+\O\left(\lambda(h + h^{-1}\eps)\beta + (\varepsilon^2 + h^2) \beta \right). \notag
\end{align*}
Divide both sides by $\rho(x)$ and use the identity
\[\rho^{-2}\div(\rho^2b \phi) =\phi(\div b + 2\nabla \log \rho \cdot b) + \nabla \varphi \cdot b\] 
to complete the proof.
\end{proof}

\section{Convergence Proofs}\label{ConvergenceProofs} 

We now prove our main results. We first need a stability estimate for the PageRank problem \eqref{PRScheme}.
\begin{lemma}[$\ell_\infty$ Stability for the PageRank Operator]\label{Stability}
Assume that $\gamma_\eps,\gamma_h\leq 1$ and $\eta<1$, where $\eta$ is defined in \eqref{AssumptionforRate} and $\gamma_\eps,\gamma_h$ in \eqref{eq:Gammas}. There exists $C,K,c>0$ such that with probability at least $1-Cn\exp\left( -cnh^{d+2}(1-\eta\gamma_\eps)^2 \right)$,  if $\eps+h\leq K(1-\eta \gamma_\eps)$ and $u,v:X_n\to \R$ satisfy
\begin{equation}\label{eq:PRstab}
u(x) - \gamma \L_n u(x) = \frac{nh^d}{d_x}v(x) \ \ \ \text{ for all }x\in X_n
\end{equation}
with $\gamma = (1-\alpha)/\alpha$, then it holds that
\begin{equation}\label{eq:stabilityestimate}
\max_{x\in X_n}|u(x)|\leq 2(1-\eta\gamma_\eps)^{-1}\max_{x\in X_n}|\rho(x)^{-1}v(x)|.
\end{equation}
\end{lemma}
\begin{proof}
We use a maximum principle argument. Let $x_0\in X_n$ be a point where $u$ attains its maximum value. Setting $\phi \equiv 1$ we have 
\[\L_nu(x_0) = \frac{1}{d_{x_0}}\sum_{y\in X_n}\omega_{yx_0}u(y) - u(x_0) \leq \frac{1}{d_{x_0}}\sum_{y\in X_n}\omega_{yx_0}u(x_0)  - u(x_0) =u(x_0) \L_n \varphi(x_0).\]
It follows that
\begin{equation}\label{eq:bound1}
\frac{d_{x_0}}{nh^d}(1 - \gamma \L_n\varphi(x_0)) u(x_0) \leq v(x_0).
\end{equation}
By Theorem~\ref{ConsistencyResult} we have
\[\frac{d_{x_0}}{\rho(x)nh^d}\L_n\varphi(x_0) =-\rho^{-2}\div(\rho^2b)\eps +\O(\lambda h + \lambda h^{-1}\eps + \eps^2 + h^3 + \eps h),\]
with probability at least $1-Cn\exp(-cnh^d \lambda^2)$. In particular, observe that $\textrm{div}(\rho^2\nabla\varphi)=0$ in Theorem \ref{ConsistencyResult}, since $\varphi\equiv 1$ is constant. Setting $\lambda=\delta h$ for $0 < \delta \leq h^{-1}$ and recalling \eqref{AssumptionforRate} we have
\[\frac{d_{x_0}}{\rho(x)nh^d}\gamma|\L_n\varphi(x)| \leq \eta\gamma_\eps + C(\gamma_h + \gamma_\eps)(\delta + h + \eps) \leq \eta\gamma_\eps + C(\delta + h + \eps),\]
with probability at least $1-Cn\exp(-cnh^{d+2} \delta^2)$. By Lemma~\ref{ConMeas1} we have
\[\frac{d_{x_0}}{\rho(x)nh^d} = 1 + \O(\eps + h)\]
 with probability at least $1-2n\exp(-cnh^{d+2})$. Inserting these observations into \eqref{eq:bound1} we have
\[\rho(x_0)(1 - \eta\gamma_\eps- C(\delta + h + \eps))u(x_0)\leq v(x_0),\]
for a constant $C>0$. Hence, selecting $\delta = (1-\eta \gamma_\eps)/(4C)$ and restricting $h+\eps \leq (1-\eta \gamma_\eps)/(4C)$, we have 
\[\frac{1}{2}\rho(x_0)(1 - \eta\gamma_\eps)u(x_0)\leq v(x_0),\]
with probability at least $1-Cn\exp(-cnh^{d+2}(1-\eta\gamma_\eps)^2)$. Therefore
\[\max_{x\in X_n}u(x) \leq 2(1-\eta\gamma_\eps)^{-1}\max_{x\in X_n}|\rho(x)^{-1}v(x)|\]
holds with probability at least $1-Cn\exp(-cnh^{d+2}(1-\eta\gamma_\eps)^2)$ provided $\eps+h\leq K(1-\eta\gamma_\eps)$. \\

For the proof of the other direction, we set $\bar{u}(x)=-u(x)$, and note that $\bar{u}$ satisfies
\[\bar{u}(x) - \gamma \L_n \bar{u}(x) = -\frac{nh^d}{d_x}v(x) \ \ \ \text{ for all } x\in X_n.\]
The argument in the first part of the proof yields
\[\max_{x\in X_n}\bar{u}(x) \leq 2(1-\eta\gamma_\eps)^{-1}\max_{x\in X_n}|\rho(x)^{-1}v(x)|,\]
which completes the proof.
\end{proof}

Given the stability estimate from Lemma~\ref{Stability}, we can now prove Theorem~\ref{Rate}.
\begin{proof}[Proof of Theorem~\ref{Rate}]
Given the assumptions on $\rho$, $b$, and $v$, the  solution $u$  of the continuum PDE \eqref{PRPDE} belongs to $C^3(\T^d)$, and $\|u\|_{C^3(\T^d)}$ depends on the ellipticity constant of the equation $\sigma_\Phi \gamma_h$ \cite{Gilbarg,Evans}. Thus, applying Lemma~\ref{ConMeas1}, and Theorem~\ref{ConsistencyResult} with $\lambda=\delta h$ yields
\begin{align*}
\frac{d_x}{\rho(x)nh^d}u(x)& - \frac{d_x}{\rho(x)nh^d}\gamma \L_nu(x)\\
&=u(x) +\gamma_\eps\rho^{-2}\div(\rho^2bu)- \frac{1}{2}\sigma_\Phi\gamma_h\rho^{-2} \div(\rho^2\nabla u) +\O\left(\delta  + \varepsilon + h \right)\\
&= \frac{v(x)}{\rho(x)}+\O\left(\delta  + \varepsilon + h \right)
\end{align*}
for all $x\in X_n$ with probability at least $1-Cn\exp\left( -cnh^{d+2}\delta^2 \right)$ for any $0 < \delta \leq h^{-1}$, where we used that $\gamma_\eps,\gamma_h\leq 1$. Therefore
\[u(x) - \gamma \L_n u(x) = \frac{nh^d}{d_x}\left(v(x) + \O(\delta + \eps + h)\right)\]
for all $x\in X_n$ with probability at least $1-Cn\exp\left( -cnh^{d+2}\delta^2) \right)$.

Consider now the difference $w_n(x) = u(x) - u_n(x)$, where $u_n$ solves the PageRank problem \eqref{PRScheme}. Then $w_n$ satisfies
\[w_n(x) - \gamma \L_n w_n(x) = \frac{nh^d}{d_x}\O(\delta + \eps + h)\]
for all $x\in X_n$ with probability at least $1-Cn\exp\left( -cnh^{d+2}\delta^2 \right)$. Applying Lemma~\ref{Stability} completes the proof.
\end{proof}

We now give the proof of Theorem~\ref{TimeRate}.
\begin{proof}[Proof of Theorem~\ref{TimeRate}]
The proof is similar to the proof of Theorem~\ref{Rate}, so we sketch the outline, omitting some details.

First, note that
\[\frac{u(x,\alpha k+\alpha)-u(x,\alpha k)}{\alpha} =u_t(x,\alpha k) + \mathcal{O}(\alpha).\]
Proceeding now as in the proof of Theorem~\ref{Rate}, we use Theorem~\ref{ConsistencyResult}, Lemma~\ref{ConMeas1}, and the observation above to deduce that
\begin{equation}\label{eq:PRdiscrete}
\frac{u(x,\alpha k+\alpha)-u(x,\alpha k)}{\alpha} +u(x,\alpha k) - \gamma\L_n u(x,\alpha k) = \frac{nh^d}{d_n(x)}v(x) + O(\lambda+\eps+h),
\end{equation}
with probability at least $1-Cn\exp(-cnh^{d+2}\lambda^2)$ for $0 < \lambda \leq 1$, where we also used that $\gamma_\eps\leq 1$. Let $\phi(x)=1$ for all $x\in X_n$. As in the proof of Lemma~\ref{Stability} we have that
\begin{equation}\label{eq:testfun}
(1-\alpha) |\L_n \varphi(x)| \leq \alpha\eta\gamma_\eps + C\alpha(\delta + h + \eps)
\end{equation}
with probability at least $1-Cn\exp(-cnh^{d+2}\delta^2)$, where $0 < \delta \leq 1$ will be chosen later. For the rest of the proof we assume the events above hold true.

Define $w_n(x,k) = u(x,\alpha k) - u_n(x,k)$. We claim that
\begin{equation}\label{eq:wnbound}
w_n(x,k) \leq Ck\alpha(\lambda + \eps + h)=:M_k.
\end{equation}
The proof of the other direction is similar, and this will complete the proof. To prove \eqref{eq:wnbound}, we use a comparision principle argument that proceeds by induction. The base case $k=0$ is trivial, since $w_n(x,0)=0$. Assume that \eqref{eq:wnbound} is true for some $k\geq 0$. Then subtracting \eqref{eq:PRdiscrete} and \eqref{eq:PRevolution} we find that $w_n$ satisfies 
\[\frac{w_n(x,k+1)-w_n(x,k)}{\alpha} +w_n(x,k) - \gamma\L_n w_n(x,k) \leq C(\lambda+\eps+h)\]
for all $x\in X_n$ and $k\geq 0$. Rearranging this we obtain
\begin{align*}
w_n(x,k+1) &\leq (1-\alpha)\left(w_n(x,k) + \L_n w_n(x,k)\right) +C\alpha(\lambda+\eps+h)\\
&=(1-\alpha)\frac{1}{d_x}\sum_{y\in X_n}w_{yx}w_n(y,k) + C\alpha(\lambda+\eps+h)\\
&\leq (1-\alpha)\frac{M_k}{d_x}\sum_{y\in X_n}w_{yx} + C\alpha(\lambda+\eps+h)\\
&= M_k(1-\alpha)\left(1 + \L_n \phi(x)\right)+ C\alpha(\lambda+\eps+h),
\end{align*}
since $w_n(x,k)\leq M_k$ for all $x\in X_n$.
Applying \eqref{eq:testfun} we have
\begin{align*}
w_n(x,k+1) &\leq M_k\left(1-(1-\eta\gamma_\eps)\alpha+ C\alpha(\delta+h+\eps)\right)+ C\alpha(\lambda+\eps+h).
\end{align*}
Hence, choosing $\delta=(1-\eta\gamma_\eps)/(4C)$ and restricting $h+\eps \leq (1-\eta\gamma_\eps)/(4C)$, we have
\begin{align*}
w_n(x,k+1) &\leq M_k + C\alpha(\lambda+\eps+h) = M_{k+1}.
\end{align*}
The claim \eqref{eq:wnbound} is thus established by induction, and this completes the proof.
\end{proof}

We now turn our attention to proving the first-order rate, which hinges on the following observation that our scheme is \textit{monotone}. 
\begin{proposition}[Monotonicity]\label{isMonotone}
Let $u,v:X_n\to \R$ and $x_0\in X_n$ such that $u(x_0)=v(x_0)$ and $u\leq v$. Then $\L_n u(x_0) \leq \L_n v(x_0)$.
\end{proposition}
\begin{proof}
The proof is immediate, since
\[\L_n u(x_0) = \frac{1}{d_x}\sum_{y\in X_n}\omega_{yx}u(y) - u(x_0) \leq\frac{1}{d_x}\sum_{y\in X_n}\omega_{yx}v(y) - v(x_0) = \L_n v(x_0).\]
\end{proof} 

In order to prove a convergence rate for the first order continuum limit, we require a Lipschitz estimate on the viscosity solution $u$ of \eqref{PROp1stOrder}. The result follows a standard maximum principle argument, which we include for completeness.
\begin{lemma}\label{lem:Lip}
Let $\rho\in C^{1,1}(\T^d)$, $b\in C^{1,1}(\T^d;\R^d)$, and $v\in C^{0,1}(\T^d)$. Assume that $\gamma_\eps \leq 1$, and $\eta < 1$. Let $u\in C(\T^d)$ be the viscosity solution of the PDE \eqref{PROp1stOrder}. If $\|Db\|_{L^\infty(\T^d)}\leq \tfrac{1}{2}(1-\eta\gamma_\eps)$ then $u\in C^{0,1}(\T^d)$ and $\|u\|_{C^{0,1}(\T^d)}$ depends only on $1-\eta\gamma_\eps$, $\|\rho\|_{C^{0,1}}$, $\|b\|_{C^{1,1}}$, and $\|v\|_{C^{0,1}}$.
\end{lemma}
\begin{proof}
Let $\delta>0$ and consider the viscosity regularized version of \eqref{PROp1stOrder}
\begin{equation}\label{eq:reg}
u_\delta +\gamma_\eps\rho^{-2}\div(\rho^2bu_\delta)  - \delta \Delta u_\delta = \rho^{-1}v \ \ \ \textrm{ on } \mathbb{T}^d.
\end{equation}
By standard elliptic PDE theory \cite{Gilbarg}, \eqref{eq:reg} has a unique solution $u_\delta\in C^{2,\nu} (\T^d)$. It is a standard result in viscosity solution theory (see, \eg, \cite{Crandall,calder2018lecture}) that $u_\delta\to u$ uniformly as $\delta\to 0$, provided $\eta \gamma_\eps < 1$. 

We now prove that the Lipschitz constant of $u_\delta$ is controlled independently of $\delta>0$. The argument is standard in elliptic PDEs, and we include it for completeness. We write $c(x) = 1 + \rho^{-2}\div(\rho^2b)$ for convenience, and note we have $c(x) \geq 1-\eta \gamma_\eps$. By the maximum principle we have
\begin{equation}\label{eq:zero}
\|u_\delta\|_{L^\infty(\T^d)}\leq (1-\eta\gamma_\eps)^{-1}\|\rho^{-1}v\|_{L^\infty(\T^d)}.
\end{equation}
To bound the gradient, we differentiate both sides of \eqref{eq:reg} in $x_i$, then multiply both sides by $u_{x_i}$, and sum over $i$ to obtain
\[c|\nabla u_\delta|^2 + u_\delta\nabla u_\delta\cdot \nabla c + \nabla u_\delta^T Db \nabla u_\delta + b\cdot \nabla |\nabla u_\delta|^2 - \sum_{i=1}^d u_{\delta,x_i}\Delta u_{\delta,x_i} = \nabla u_\delta\cdot \nabla (\rho^{-1}v).\]
We use the identity
\[w\Delta w = \frac{1}{2}\Delta w^2 - |\nabla w|^2 \leq \frac{1}{2}\Delta w^2\]
with $w=u_{\delta,x_i}$ to obtain
\[c|\nabla u_\delta|^2 + u_\delta\nabla u_\delta\cdot \nabla c + \nabla u_\delta^T Db \nabla u_\delta + b\cdot \nabla |\nabla u_\delta|^2 -\frac{\delta}{2}\Delta |\nabla u_\delta|^2 \leq \nabla u_\delta\cdot \nabla (\rho^{-1}v)\]
on the torus $\T^d$. Now, let $x_0\in \T^d$ be a point where $|\nabla u_\delta|^2$ attains its maximum value on $\T^d$. Then we have $\nabla |\nabla u_\delta(x_0)|^2 = 0$ and $\Delta |\nabla u_\delta(x_0)|^2\leq 0$, which yields
\begin{align*}
(1-\eta\gamma_\eps) |\nabla u_\delta(x_0)|^2&\leq c(x_0)|\nabla u_\delta(x_0)|^2\\
&\leq  - u_\delta\nabla u_\delta\cdot \nabla c - \nabla u_\delta^T Db \nabla u_\delta + \nabla u_\delta\cdot \nabla (\rho^{-1}v)\vert_{x_0}\\
&\leq C(\|u_{\delta}\|_{L^\infty(\T^d)}+1)|\nabla u_\delta(x_0)| + \|Db\|_{L^\infty(\T^d)}|\nabla u(x_0)|^2,
\end{align*}
where $C$ depends on $\rho, c$, and $v$. Since $\|Db\|_{L^\infty(\T^d)}\leq (1 - \eta\gamma_\eps)/2$, we can apply Cauchy's inequality with $\eps$ to obtain
\[|\nabla u_\delta(x_0)|^2 \leq C(1-\eta\gamma_\eps)^{-2}(\|u_{\delta}\|_{L^\infty(\T^d)}+1)^2,\]
which completes the proof.
\end{proof}

We now give the proof of Theorem~\ref{1stRate}. The proof follows the method of doubling the variables, which is used in viscosity solution theory for proving the comparison principle and establishing error estimates \cite{Crandall,calder2018lecture}. For the reader's convenience, we review the definition of viscosity solution in Section~\ref{sec:viscosity}.
\begin{proof}[Proof of Theorem~\ref{1stRate}]
First, by Lemma~\ref{Stability} we have that $u_n$ is uniformly bounded, independent  of $n$, with probability at least $1 - Cn \exp\left( -cnh^{d+2}(1-\eta\gamma_\eps)^2 \right)$ provided $\eps+h$ is sufficiently small. We assume these conditions throughout the rest of the proof.

Define the doubling-of-variables function 
\begin{align*}
\Phi(x,y) := u_n(x) -u(y) - \frac{\theta}{2}|x-y|^2 \textrm{ on } X_n \times \mathbb{T}^d, 
\end{align*} 
which has a maximum at $(x_n, y_n) \in X_n \times \mathbb{T}^d$. Since $\Phi(x_n, y_n) \geq \Phi(x_n, x_n)$, we have 
\begin{align*}
u_n(x_n) - u(y_n) -\frac{\theta}{2}|x_n - y_n|^2 &\geq u_n(x_n) - u(x_n).
\end{align*}
By Lemma~\ref{lem:Lip}, $u$ is Lipschitz and so
\[\frac{\theta}{2} |x_n-y_n|^2 \leq u(x_n) - u(y_n) \leq C|x_n-y_n|.\]
 Hence we have the bound $|x_n-y_n| \leq \displaystyle\frac{C}{\theta}.$

Define $\psi(x) = \frac{\theta}{2}|x-y_n|^2$ and $\xi_n = u_n(x_n) - \frac{\theta}{2}|x_n-y_n|^2$. By the definition of $(x_n,y_n)$, $u_n - \psi$  has a maximum at $x_n$ relative to $X_n$. Since we have $u_n(x_n)=\psi(x_n)+\xi_n$ we see that $u_n\leq \psi + \xi_n$ on $X_n$ and so by Proposition~\ref{isMonotone} we have $\L_n u_n(x_n) \leq \L_n (\psi + \xi_n)(x_n)$. It follows that
\begin{align*}
\frac{nh^d}{d_{x_n}}v(x_{x_n})&=u_n(x_n) - \gamma \L_n u_n(x_n)\\
&\geq u_n(x_n) - \gamma \L_n (\psi+\xi_n)(x_n).
\end{align*}
By Theorem~\ref{FirstConsistencyResult}, Lemma~\ref{ConMeas1} we have
\begin{equation}\label{eq:onedir}
\rho(x_n)^{-1}v(x_n)\geq u_n(x_n) + \gamma_\eps(\rho^{-2}\div(\rho^2 b)u_n + \nabla \psi\cdot b)\big\vert_{x_n} - C\theta(\delta + \eps + \gamma_h),
\end{equation}
with probability at least $1 - Cn \exp\left( -cnh^{d+2}\delta^2 \right)$.

We now define $\phi(y) = -\frac{\theta}{2}|x_n-y|^2$.  Then $u-\phi$ attains its minimum over $\T^d$ at $y_n$, and so by the definition of viscosity supersolution we have
\begin{equation}\label{eq:otherdir}
\rho(y_n)^{-1}v(y_n)\leq u(y_n) + \gamma_\eps(\rho^{-2}\div(\rho^2 b)u + \nabla \phi\cdot b)\big\vert_{y_n}.
\end{equation}
Write $c(x) = 1 + \gamma_\eps \rho^{-1}\div(\rho^2 b)$ and note that $\nabla \psi(x_n)=\nabla \phi(y_n) = \theta(x_n-y_n)$ is bounded. Now subtract \eqref{eq:onedir} from \eqref{eq:otherdir} to find that
\[c(x_n)u_n(x_n) - c(y_n)u(y_n) \leq C|x_n-y_n| + C\theta(\delta+\eps+\gamma_h),\]
where $C>0$ depends on the Lipschitz constants of $\rho$, $v$, and $b$, and the positive lower bound for $\rho$. We now write 
\[c(x_n)u_n(x_n) - c(y_n)u(y_n) = c(x_n)(u_n(x_n) - u(y_n)) + (c(x_n) - c(y_n))u(y_n)\]
to obtain
\[c(x_n) (u_n(x_n) - u(y_n)) \leq C|x_n-y_n|+C\theta(\delta+\eps+\gamma_h),\]
where $C$ depends additionally on $\|u\|_\infty$. Since $c(x_n) \geq 1 - \eta\gamma_\eps>0$ and $|x_n-y_n|\leq C/\theta$ we have
\[\max_{x\in X_n} (u_n(x)-u(x)) \leq u_n(x_n) - u(y_n) \leq \frac{C}{\theta}+C\theta(\delta+\eps+\gamma_h).\]
Optimizing over $\theta>0$ yields
\[\max_{x\in X_n} (u_n(x)-u(x)) \leq C\sqrt{\delta+\eps+\gamma_h}.\]

The proof of the bound in the other direction on $u-u_n$ is analogous,  except we use the auxiliary function 
\[\Phi(x,y) = u(x) - u_n(y) - \frac{\theta}{2}|x-y|^2. \]
\end{proof}

\section{Numerical Experiments}\label{Numerical}

We now turn to some brief numerical experiments to illustrate our main results. We present numerical experiments confirming the convergence rates and parameter scalings from Theorems \ref{Rate} and \ref{1stRate} in Section  \ref{sec:rate}, and in Section \ref{sec:median} we give an application of PageRank to data depth and high dimensional medians. The code for the numerical experiments was implemented in Python $3.7.9$ and is available on GitHub: \url{https://github.com/jwcalder/ContinuumPageRank}.

\subsection{Convergence rates and parameter scalings}
\label{sec:rate}

To test the linear convergence rates in our main result Theorem \ref{Rate}, we work with an explicit solution of the continuum PDE \eqref{PRPDE}. We take the $2$-dimensional torus $\T^2$ and set
\[u(x) = 2 - (\cos(2\pi x_1) + \cos(2\pi x_2)),\]
and
\[v(x) = 2 - \left( 1 + \tfrac12 \gamma_h\pi^2 \right)(\cos(2\pi x_1) + \cos(2\pi x_2)),\]
where $\gamma_h$ is given in \eqref{eq:Gammas}. This pair of functions satisfies the PDE \eqref{PRPDE} with  $b=0$, $\sigma_\Phi = \tfrac14$ and $\rho=1$.  The corresponding discrete PageRank problem is to take the points $X_n$ to be an \emph{i.i.d.}~sample of size $n$ uniformly distributed on the torus $\T^2$, and construct a geometric graph over the data with edge weights given by \eqref{eq:weights} with $b(x)=0$, $\phi(t) = \tfrac{1}{\pi}$ for $0 \leq t \leq 1$, and $\Phi(t)=0$ otherwise. 

\begin{figure}[!t]
\centering
\subfloat[Synthethic data]{\includegraphics[clip=true,trim=10 15 15 15, width=0.49\textwidth]{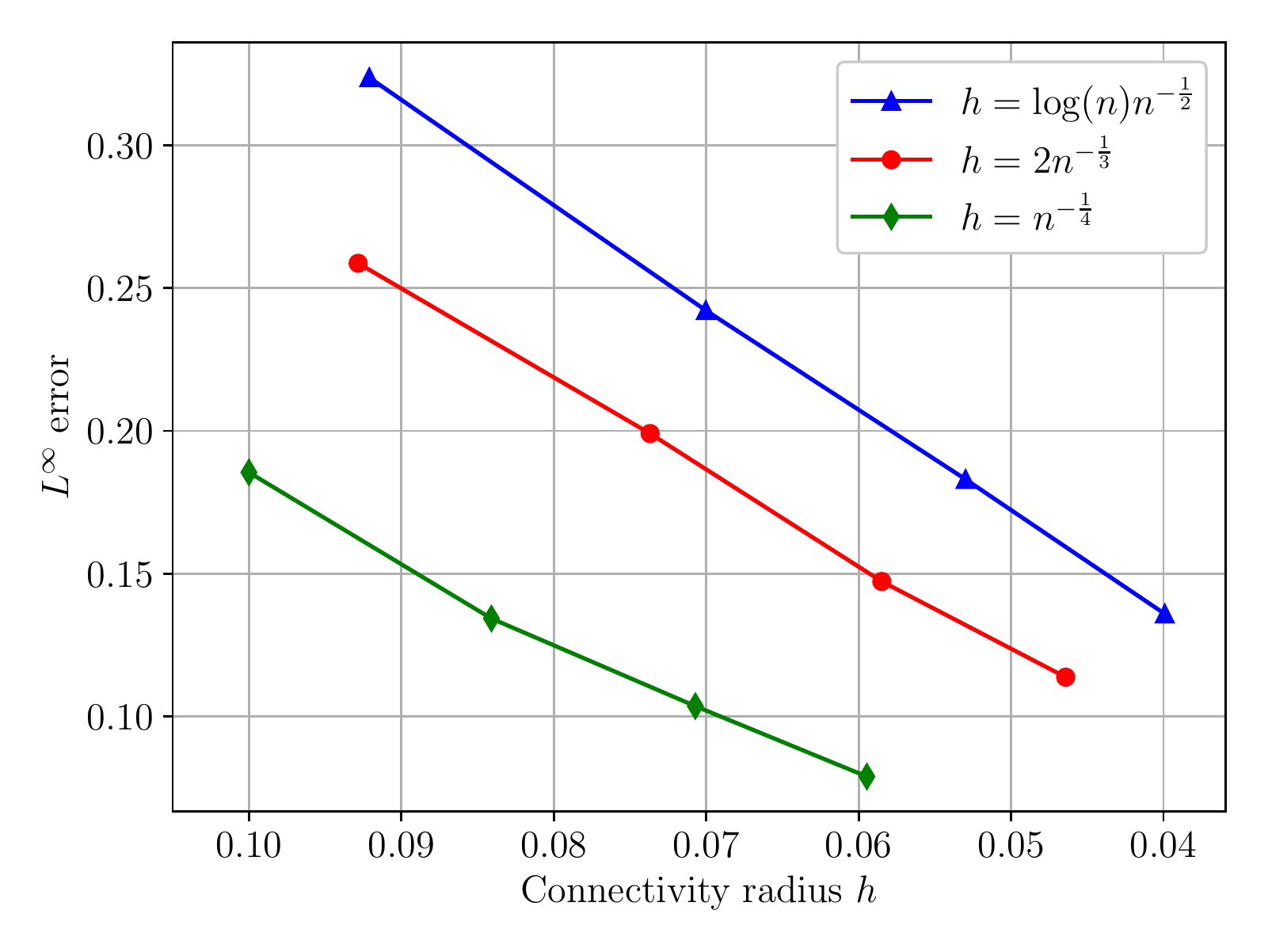}\label{fig:rate}}\hfill
\subfloat[MNIST data]{\includegraphics[clip=true,trim=10 15 15 15, width=0.49\textwidth]{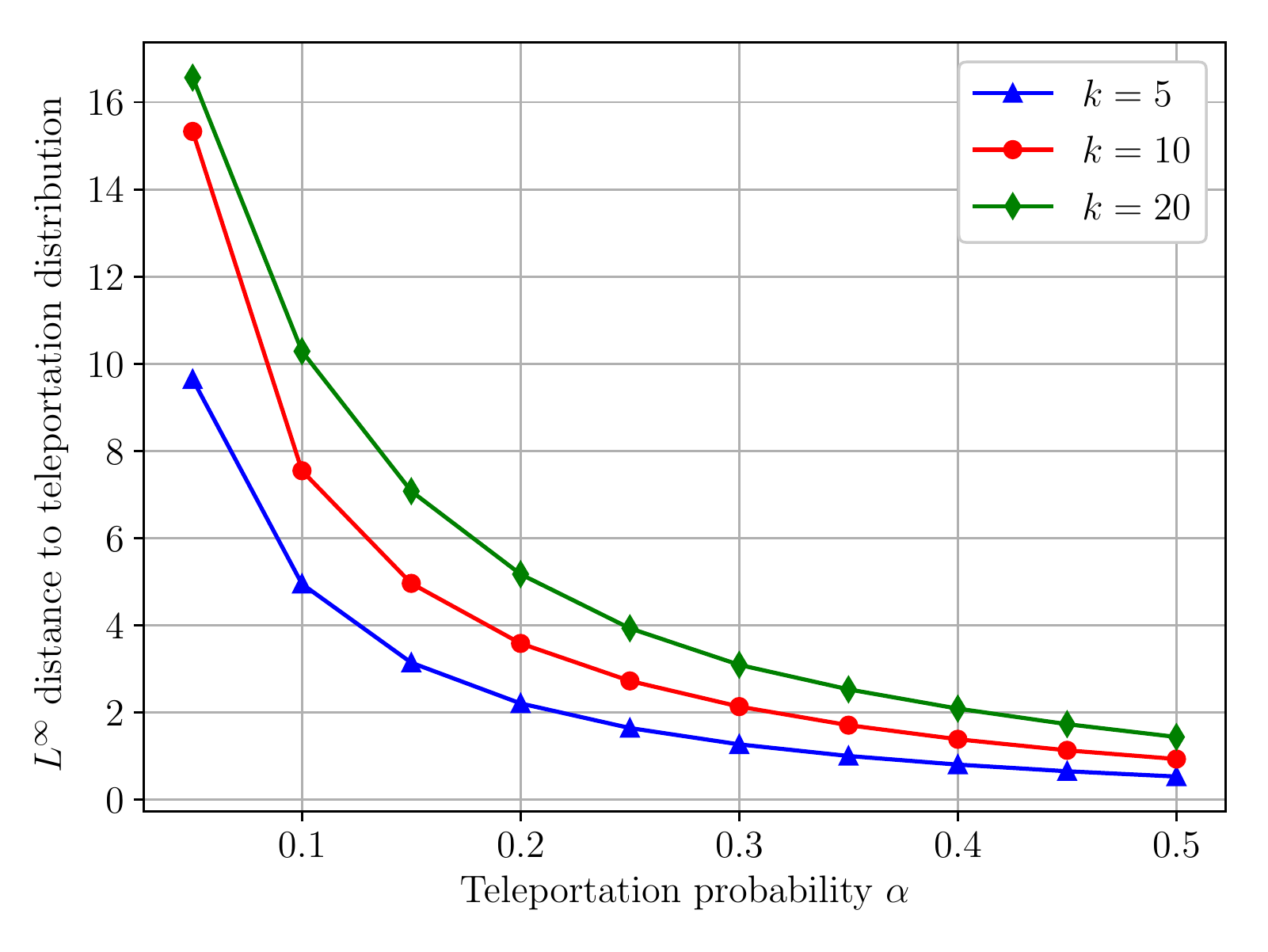}\label{fig:alpha}}
\caption{\textbf{Left:} For an explicit solution to the continuum limit PDE \eqref{PRPDE}, we plot the connectivity radius $h$ vs. the $L^\infty$ norm of the solution error as the number of nodes increases for the three different dependencies $h = h(n)$ noted in the legend. The experiment demonstrates the linear convergence rate in our main result Theorem \ref{Rate}.\\ \textbf{Right:} A computational experiment on MNIST data illustrating, for different numbers of neighbors $k$, how the $L^\infty$ norm to the uniform teleportation distribution depends on the transportation probability $\alpha$.}
\label{fig:exp}
\end{figure}

We ran experiments with  $n=10^4,2\times 10^4, 4\times 10^4,$ and $n=8\times 10^4$, and three different choices for the connectivity length scale $h$: $h=\log(n)n^{-\frac12}$, $h=2n^{-\frac13}$, and $h=n^{-\frac14}$. To keep $\gamma_h$ roughly constant, we set $\alpha = Ch^2$ in each case, with $C=30,20,$ and $10$, respectively. We ran 100 trials for each combination of $n$, $h$, and $\alpha$, and computed the average $L^\infty$ error over all trials. The experiments were conducted on a $24$-core machine with $3.8$ Gz processors and $256$ GB of RAM. For the largest number of vertices, $n = 8\times 10^4$, each trial took approximately six minutes on a single core and used $8$ GB of RAM. 

Theorem \ref{Rate} guarantees a linear $O(h)$ convergence rate for a much larger length scale $h \gg   \log(n)^{\frac16}n^{-\frac16}$ (due to \eqref{eq:almostsure_rate}), but the graphs are too dense at this length scale to perform experiments with large $n$.  Nevertheless, we observe linear $O(h)$ convergence rates at all three smaller length scales, as is shown in the plot in Figure \ref{fig:rate}. It would be an interesting problem for future work to prove the linear rate from Theorem \ref{Rate} at these smaller length scales.

To test how the parameter scalings suggested by our main results transfer to real data, we ran an experiment with the MNIST dataset \cite{lecun1998gradient} to compare how the distance between the PageRank vector and the teleportation distribution depends on $\alpha$. The MNIST dataset contains 70,000 $28\times 28$ pixel grayscale images of handwritten digits 0--9 (see Figure \ref{fig:PageRankMNIST} for an example of some MNIST images).  We used all  $70,000$ images to construct the graph, and flattened the images into vectors in $\R^{784}$. We used a $k$-nearest neighbor graph with edge weight between images $i, j$ is given by 
\begin{align*}
w_{i,j} = \exp\left(\frac{-4|x_i - x_j|^2}{d_k(x_i)^2}\right),
\end{align*}
where $d_k(x_i)$ is the Euclidean distance between image $x_i$ and its $k^{\rm th}$ nearest neighbor. A $k$-nearest neighbor graph is naturally directed (i.e., asymmetric), since the $k$-nearest neighbor relation is not symmetric. We use the uniform teleportation probability distribution $v$.

The continuum PDE \eqref{PRPDE} suggests that the difference between the PageRank vector $u$ and the teleportation distribution $v$ is proportional to $\gamma_n = (\alpha^{-1}-1)h^2$. In Figure \ref{fig:alpha} we show the $L^\infty$ distance between the teleportation distribution and PageRank vector as a function of $\alpha$ for $k$-nearest neighbor graphs with $k=10,20,30$. The relationship between the error and $\alpha$ is well-approximated by the predicted scaling of $\alpha^{-1}$ (a regression yields a power law $\alpha^{-p}$ with $p=1.21,1.27$ and $p=1.05$ for $k=10,20$ and $k=30$, respectively). We also see that  the error is increasing with $k$, as expected, since $h$ grows with $k$. In practice, it is important to choose $\alpha$ small enough so that the PageRank vector does not simply copy the teleportation distribution, and instead explores the geometry of the graph, while not choosing $\alpha$ so small that the PageRank iterations are slow to converge.

We also mention that we ran the same experiment on the Fashion MNIST dataset \cite{xiao2017fashion}, and obtained very similar results. The Fashion MNIST dataset is a drop-in replacement for MNIST, where the classes are types of clothing, and each data point is a $28\times 28$ pixel grayscale image of a clothing item from a fashion catalog (see Figure \ref{fig:FashionMNIST} for an example of some Fashion MNIST images).

\subsection{PageRank for data depth}
\label{sec:median}

In this section, we present an application of using PageRank for computing a notion of data depth and high dimensional medians. We ran experiments on the MNIST and Fashion MNIST datasets, using the same graph construction as described in Section \ref{sec:rate}, with $k=10$ Euclidean nearest neighbors.  We used the localized PageRank algorithm, focused on each class separately, with $\alpha=0.05$. That is, we ran 1 experiment for each class with the teleportation distribution taken as the characteristic function of that class. This generates a ranking of all images relative to each class, which can be interpreted as a notion of in-class data depth, with the highest ranked images corresponding to a generalization of median images.

In Figures \ref{fig:PageRankMNIST} and \ref{fig:FashionMNIST} we show the top 11 ranked images for each class (on the right), alongside a random selection of images (on the left).  We note that the highest ranked images are the most ``typical'' for each class, which can be interpreted as median images. There is very little variation amongst the highest ranked digits, as compared to random digits. The same is observed for Fashion MNIST; the ranked images display less variation, have fewer patterns, and are more solid-colored. 

\begin{figure}[t]
\includegraphics[width=0.48\textwidth,trim=0 0 112 0, clip=true]{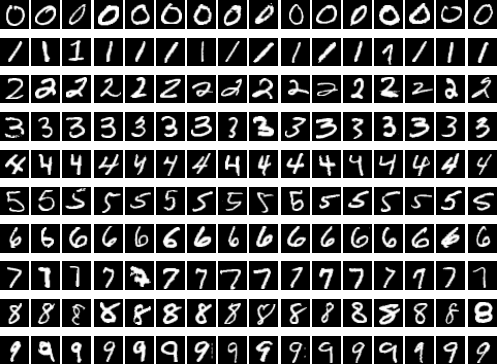}\hfill
\includegraphics[width=0.48\textwidth,trim=0 0 112 0, clip=true]{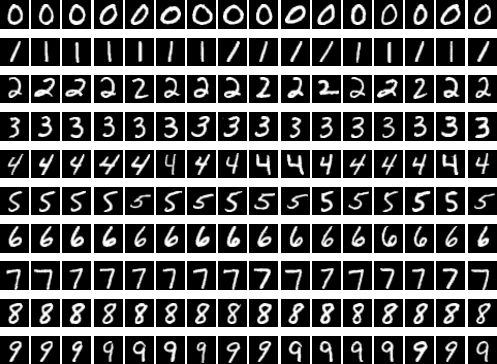}
\caption{\textbf{Left:} A selection of $11$ random samples per each digit from the MNIST dataset. \\ \textbf{Right:} The highest ranked images per digit, using localized PageRank with $\alpha = 0.05$.}
\label{fig:PageRankMNIST}
\centering
\end{figure}

\begin{figure}[t]
\includegraphics[width=0.48\textwidth,trim=0 0 112 0, clip=true]{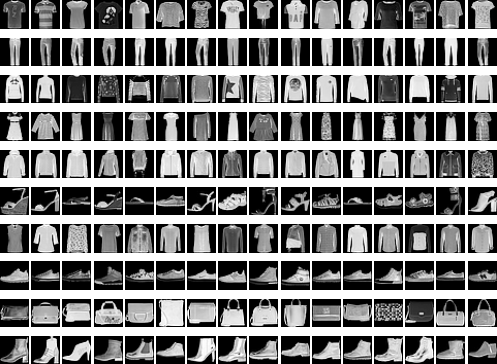}\hfill
\includegraphics[width=0.48\textwidth,trim=0 0 112 0, clip=true]{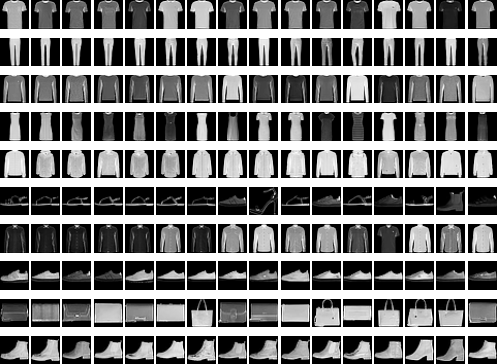}
\caption{\textbf{Left:} A selection of $11$ random images per class from the Fashion MNIST dataset, separated by type (i.e., handbag, shirt, pants, etc.). \\ \textbf{Right:} The highest ranked images per type of item, using localized PageRank with $\alpha = 0.05$. Notice that the highest ranked images are more pattern-free when compared to the random images.} 
\label{fig:FashionMNIST}
\centering
\end{figure}

\section{Conclusion}

In this paper, we established a new framework for rigorously studying continuum limits for discrete learning problems on directed graphs. We introduced the \emph{random directed geometric graph}, and showed how to prove a continuum limit with this graph model for the PageRank algorithm. Depending on parameter scalings, the continuum limit is either a second order reaction-diffusion-advection equation, or a first order reaction-advection equation, whose solution can be interpreted in the viscosity sense.  We also presented some numerical experiments confirming our theoretical results, and exploring applications of the PageRank algorithm to multi-variate data depth.

\begin{acknowledgements}
A.~Yuan and J.~Calder were partially supported by NSF-DMS grants 1713691 and 1944925, and a University of Minnesota Grant in Aid award. J.~Calder was also partially supported by the Alfred P.~Sloan Foundation. B.~Osting was partially supported by NSF DMS 16-19755 and DMS 17-52202.
\end{acknowledgements}

\appendix

\section{A more general operator}\label{GeneralizedB}

In this section, we consider the more general case where the matrix $B$ in the definition of the weights for a random directed geometric graph (see Section~\ref{sec:RDGG} for definitions) is not the identity matrix. This case turns out to be difficult to interpret, and so we omit the details of extending our main results to this setting, but include the discussion below for completeness.

Assume that $B(x) \in \mathbb{R}^{d\times d}$ has bounded, Lipschitz continuous entries and has full rank for all $x\in \T^d$. Then, following the arguments of the previous subsection, the continuum limit operator is 
\begin{align*}
L^{B}u(x) &:=-\gamma_h \left(\sigma_\Phi \textrm{Tr}\left(B(x)^{-1}B(x)^{-1} \nabla^2u(x)\right)\hspace{-1mm}+ \hspace{-1mm}\sum_{i,j} u_{x_i x_j}(x) \hspace{-1mm}\int \Phi(|z|) f(x,z)_i f(x,z)_j dz \right) \\
&\hspace{3in}+ \sum_i \mathfrak{b}^i(x) u_{x_i}(x) + \mathfrak{c}(x) u(x), 
\end{align*}
where the first order term is
\begin{align*}
&\sum_i \mathfrak{b}^i(x) u_{x_i}(x) \\
&= \gamma_\varepsilon \nabla u(x)\cdot \nabla b(x) - \gamma_h \left(\sigma_\Phi\left(B(x)^{-1}\right)^2 \nabla\log\rho(x) - F\right)\cdot \nabla u(x) \\
&- \gamma_h \int \Phi(|z|) \left(\nabla u(x) \cdot f(x,z)\right) \left(\nabla \log \rho(x) \cdot f(x,z)\right) dz \\ 
&+ 2\gamma_h \sum_i u_{x_i}(x) \int \Phi(|z|) \left(\left(B(x)^{-1}z\right)_i +f(x,z)_i\right)\textrm{Tr}\left(B(x)^{-1}DB(x)B(x)^{-1}z\right) dz,
\end{align*}
and the zeroth order term is
\begin{align*}
\mathfrak{c}(x)u(x) &= \Big(1+ \gamma_\varepsilon \left(\rho(x)^{-2}\textrm{div}(\rho(x)^2 b(x)) - \textrm{Tr}\left(B(x)^{-1}DB(x)b(x)\right) \right) \\
&\hspace{2in}+ \gamma_h \left(\nabla\log\rho(x) \cdot F - H_\Phi\right)\Big) u(x). 
\end{align*} 

We use $f(x,z)$ to represent the vector \[\left(B(x)^{-1}DB(x)B(x)^{-1}z\right)x\] and  $F$ to represent the vector \[B(x)^{-1} DB(x) B(x)^{-1}zB(x)^{-1}DB(x)xB(x)^{-1}z - \left(B(x)^{-1}DB(x)B(x)^{-1}z\right)^2x.\]

To clarify the notation, we use $DB(x)$ to represent the tensor where the $ij$th term is $\nabla B^{ij}(x)$. For instance, the $ij$th component of the term $DB(x) B(x)^{-1}z$ is the dot product of the gradient of the $ij$th component of $B$ with the vector $B(x)^{-1}z$, \ie,
\begin{align*}
\nabla B^{ij}(x) \cdot B(x)^{-1}z.
\end{align*}

In the zeroth order term, we use $H_\Phi$ to represent the integral 
\begin{align*}
&\int \Phi(|z|)\Bigg(2\textrm{Tr}^2\left(B(x)^{-1}DB(x)B(x)^{-1}z\right)-\textrm{Tr}\left(\left(B(x)^{-1}DB(x)B(x)^{-1}z\right)^2\right) \nonumber \\
&\hspace{.5cm} -\frac{1}{2}\textrm{Tr}\left(B(x)^{-1}(B(x)^{-1}zD^2B(x)B(x)^{-1}z)\right) +B(x)^{-1}DB(x)B(x)^{-1}zB(x)^{-1}z \nonumber\\
&\hspace{.5cm}+ B(x)^{-1}D^2B(x)zB(x)^{-1}z -B(x)^{-1}DB(x) B(x)^{-1}DB(x) B(x)^{-1}zB(x)^{-1}z \\
&\hspace{.5cm} -2\nabla\log\rho(x)\cdot f(x,z) \textrm{Tr}\left(B(x)^{-1}DB(x)B(x)^{-1}z\right) \Bigg)dz.
\end{align*}

Due to the presence of $DB(x)$, the terms in $F$, $f$, and $H_\Phi$ will vanish when $B$ is a constant matrix, such as the identity.

\section{Definition of viscosity solution}
\label{sec:viscosity}

Viscosity solutions are a notion of weak solution for PDEs based on the maximum principle. Viscosity solutions enjoy very strong stability and uniqueness properties, and the theory is especially useful for passing from discrete to continuum limits (see, \eg\cite{calder2018lecture,calder2018limit,calder2014hamilton}). We review here the basic definitions of viscosity solutions of the first order equation
\begin{equation}\label{eq:HJB}
H(x,u,\nabla u) = 0 \ \ \ \text{in }\Omega,
\end{equation}
where $\Omega\subset \R^d$ is open. For viscosity solution on the torus $\T^d$, we take $\Omega = \R^n$ and treat functions on  $\T^d$ as $\Z^d$-periodic functions on $\R^d$ for defining viscosity solutions.
\begin{definition}\label{def:visc_bc}
We say $u \in C(\bar{\Omega})$ is a \emph{viscosity subsolution} of \eqref{eq:HJB} if for all $x \in \Omega$ and every $\phi \in C^\infty(\R^n)$ such that $u-\phi$ has a local maximum at $x$ we have
\[H(x,u(x),\nabla \varphi(x))\leq 0.\]

Likewise, we say that $u \in C(\bar{U})$ is a \emph{viscosity supersolution} of \eqref{eq:HJB} if for all $x \in \Omega$ and every $\phi \in C^\infty(\R^n)$ such that $u-\phi$ has a local minimum at $x$ we have
\[H(x,u(x),\nabla \varphi(x))\geq 0.\]

Finally, we say that $u$ is a \emph{viscosity solution} of \eqref{eq:HJB} if $u$ is both a viscosity sub- and supersolution. 
\end{definition}

\end{document}